\documentclass[11pt,letter,leqno]{amsart}
\usepackage{amsmath,amsfonts,amssymb,amsthm}
\usepackage[ruled]{algorithm2e}
\usepackage{graphicx, color}
\usepackage{cite}

\usepackage{accents}
\newlength{\dwidehatheight}
\newcommand{\doublehat}[1]{%
    \settoheight{\dwidehatheight}{\ensuremath{#1}}%
    \addtolength{\dwidehatheight}{-0.25ex}%
    \widehat{\vphantom{\rule{1pt}{\dwidehatheight}}%
    \smash{\widehat{#1}\hspace{2pt}}}}

\theoremstyle{plain}

\newtheorem{theorem}{Theorem}
\newtheorem{lemma}{Lemma}

\theoremstyle{definition}
\newtheorem{example}{Example}
\newtheorem{remark}{Remark}

\newtheorem{definition}{Definition}

\newtheorem{claim}{Claim}

\numberwithin{equation}{section}

\DeclareMathOperator{\Ima}{Im}

\DeclareMathOperator{\vcdim}{VC-dim}

\newcommand{\uX}{\ensuremath{\underline{X}}\xspace}
\newcommand{\uY}{\ensuremath{\underline{Y}}\xspace}
\newcommand{\U}{\ensuremath{U}\xspace}

\newcommand{\M}{\ensuremath{\mathcal{M}}}

\newcommand{\covectors}{\ensuremath{\mathcal{L}}}
\newcommand{\cocircuits}{\ensuremath{\mathcal{C}^*}}
\newcommand{\circuits}{\ensuremath{\mathcal{C}}}
\newcommand{\topes}{\ensuremath{\mathcal{T}}}

\newcommand{\rk}{\ensuremath{{\rm rank}}}

\newcommand{\product}{\square}
\newcommand{\osc}{\ensuremath{{\rm osc}}}
\newcommand{\cross}{\ensuremath{{\rm cross}}}

\newcommand{\hS}{\widehat{S}}
\newcommand{\hhS}{\doublehat{S}}
\newcommand{\hhQ}{\doublehat{Q}}

\newcommand{\ttT}{\ensuremath{{\widetilde{T}}}}

\newcommand{\Q}{\mathrm Q}
\renewcommand{\SS}{\ensuremath{\mathcal{S}}}
\newcommand{\hQ}{\ensuremath{{\widehat{Q}}}}

\newcommand{\OM}{\mathrm{OM}}

\newcommand{\ReLU}{\mathrm{ReLU}}
\newcommand{\Sam}{\mathrm{\bf Samp}}

\DeclareMathOperator{\pr}{\rm pr}

\DeclareMathOperator{\cube}{\mathrm Q}
\DeclareMathOperator{\Sep}{\mathrm Sep}

\DeclareMathOperator{\face}{\mathrm{F}}

\DeclareMathOperator{\vcd}{VC-dim}
\renewcommand{\SS}{\ensuremath{\mathcal{S}}}

\definecolor{color1}{RGB}{255,128,0}

\usepackage[left=1in,right=1in,top=1in,bottom=1in,foot=0.5in]{geometry}

\begin{document}
	
	\centerline{\bf\Large Labeled sample compression schemes}
	\centerline{\bf\Large  for complexes of oriented matroids}

	\bigskip
	\centerline{\sc \large Victor Chepoi$^{1}$, Kolja Knauer$^{1,2}$, and Manon
	Philibert$^{1}$}
	
	\bigskip
	\centerline{$^{1}$LIS, Aix-Marseille Universit\'e, CNRS, and Universit\'e
	de Toulon}
	\centerline{Facult\'e des Sciences de Luminy, F-13288 Marseille Cedex 9,
	France}
	\centerline{ {\sf
	\{victor.chepoi,kolja.knauer,manon.philibert\}@lis-lab.fr} }
	
	\bigskip
	\centerline{$^{2}$Departament de Matem\`atiques i Inform\`atica,
	Universitat de Barcelona (UB),}
	\centerline{Barcelona, Spain}



\bigskip\medskip
\noindent
{\footnotesize {\bf Abstract}
	We show that the topes of a complex of oriented matroids (abbreviated COM)
	of VC-dimension $d$ admit a proper labeled sample compression scheme of
	size $d$. This considerably extends results of Moran and Warmuth on ample classes, of
Ben-David and Litman on affine arrangements of hyperplanes,
    and of the authors on complexes of uniform oriented matroids, and is a step
    towards the sample compression conjecture -- one of
	the oldest open problems
	in computational learning theory. On the one hand, our approach exploits
	the rich combinatorial cell structure of COMs via oriented matroid theory.
	On the other hand, viewing tope graphs of COMs as partial cubes creates a
	fruitful link to metric graph theory.}

\section{Introduction}

\subsection{General setting} Littlestone and Warmuth~\cite{LiWa} introduced sample compression schemes as an
abstraction of the underlying structure of learning algorithms. Roughly,
the aim of a sample compression scheme is to compress samples of a
\emph{concept class} (i.e., of a set system) $\mathcal C$ as much  as possible, such that data coherent
with the original samples can be reconstructed from the compressed data.
%
There are two
types of sample compression schemes: labeled, see~\cite{LiWa,FlWa} and unlabeled, see
\cite{Fl,BDLi,KuWa}. A {labeled compression scheme}  of size $k$ compresses
every sample of $\mathcal C$ to a labeled subsample of size at most $k$ and  an
{unlabeled compression scheme} of size $k$ compresses every sample  of
$\mathcal C$ to a subset of size at most $k$ of the domain of the sample (see the
end of the introduction for precise definitions). The Vapnik-Chervonenkis
dimension (\emph{VC-dimension}) of a set system, was introduced by~\cite{VaCh}
as a complexity measure of set systems. VC-dimension
is central in PAC-learning and  plays an important role in combinatorics,
algorithmics, discrete geometry, and combinatorial optimization. In particular,
it coincides with the rank in the theory of (complexes of) oriented matroids.
Furthermore, within machine learning and closely tied to the topic of this
paper, the \emph{sample compression conjecture} of
\cite{FlWa} and~\cite{LiWa} states that \textit{any set system of VC-dimension
$d$ has a labeled
sample compression scheme of size $O(d)$.} This question remains one of
the oldest open problems in computational learning theory.

\subsection{Related work} The best-known
general upper bound is due to Moran and Yehudayoff~\cite{MoYe} and shows that
there exist labeled compression schemes of size $O(2^d)$ for any set system of
VC-dimension $d$. The labeled compression scheme of~\cite{MoYe} is not proper (i.e., does not
necessarily return a set from the input set system)
and
it is even open if there exist proper labeled sample compression schemes which
compress samples with support larger than $d$ to
subsamples with strictly smaller support~\cite{Mo_pers}.
From below, Floyd and
Warmuth~\cite{FlWa} showed that there are classes of VC-dimension $d$
admitting no labeled compression scheme of size less than $d$ and that no concept class of VC-dimension $d$ admits a labeled compression scheme of size at most $\frac{d}{5}$.
P\'alv\"olgyi and Tardos~\cite{PaTa}
exhibited a concept class of VC-dimension $2$ with no unlabeled
compression scheme of size $2$. However, no similar results are known 
for labeled sample compression schemes. Prior to~\cite{PaTa}, it was shown in~\cite{Ney} that the concept class of
positive halfspaces in ${\mathbb R}^2$ (which has VC-dimension 2) does not admit proper unlabeled sample compression schemes of size 2.

For more structured concept classes better upper bounds are known.
Ben-David and Litman\cite{BDLi} proved a compactness lemma, which reduces the existence
of labeled or unlabeled compression schemes for arbitrary concept classes to
finite concept classes. They also obtained unlabeled compression schemes for regions in arrangements of affine hyperplanes
(which correspond to realizable affine oriented matroids in our language). Finally, they obtained
sample compression schemes for concept classes by embedding them into concept classes
for which such schemes were known. Helmbold, Sloan, and Warmuth~\cite{HeSlWa} constructed unlabeled compression
schemes of size $d$ for intersection-closed concept classes of VC-dimension $d$. They compress each
sample to a minimal generating set 
and show that the size of this set is upper bounded by the VC-dimension. An
important class for which positive results are available is given by ample
set systems~\cite{BaChDrKo,Dr} (originally introduced as lopsided sets by Lawrence~\cite{La}). They
capture an important variety of combinatorial objects, e.g., (conditional)
antimatroids, see~\cite{EdJa}, diagrams of (upper locally) distributive lattices,
median graphs or CAT(0) cube complexes, see~\cite{BaChDrKo} and were rediscovered in
various disguises, e.g. by~\cite{BoRa} as \emph{extremal for (reverse) Sauer} and
by~\cite{MeRo} as \emph{shattering-extremal}~\cite{MeRo}. Moran and Warmuth~\cite{MoWa} provide
labeled sample compression schemes of size $d$ for ample set systems of
VC-dimension $d$. For maximum concept classes (a subclass of ample set systems)
unlabeled sample compression schemes of size $d$ have been designed by Chalopin
et al. ~\cite{ChChMoWa}. They also characterized unlabeled
compression schemes for ample classes via the existence of
\emph{unique sink orientations} of their graphs. However, the existence of such orientations remains open.

\subsection{OMs and COMs}\label{ss:OM-COM}
A structure somewhat opposed to ample classes are Oriented Matroids
(OMs), see the book of Bj\"{o}rner et al.~\cite{BjLVStWhZi}. Co-invented by Bland and Las Vergnas~\cite{BlLV}
and Folkman and Lawrence~\cite{FoLa}, and further investigated  by Edmonds and Mandel~\cite{ed-ma-82} and
many other authors, oriented matroids  represent a unified combinatorial theory of orientations of ordinary matroids,
which simultaneously captures the basic properties of sign vectors representing the regions in a hyperplane
arrangement  in ${\mathbb R}^d$ and of  sign vectors of the circuits  in a directed
graph. OMs provide a framework for the analysis of
combinatorial properties of geometric configurations occurring  in discrete
geometry and in machine learning. Point and vector configurations, order types,
hyperplane and pseudo-line arrangements, convex polytopes, directed graphs, and
linear programming find a common generalization in this language. The
Topological Representation Theorem of~\cite{FoLa} connects the theory of OMs on a
deep level to arrangements of pseudohyperplanes and distinguishes it from the
theory of ordinary matroids.

Complexes of Oriented Matroids (COMs) were introduced
by Bandelt, Chepoi, and Knauer~\cite{BaChKn} as a natural common generalization of ample classes and OMs.
Ample classes are exactly the COMs with cubical cells, while OMs are
the COMs with a single cell. In general COMs, the cells are OMs and the
resulting
cell complex is contractible. In the realizable setting, a COM corresponds to
the intersection pattern of a hyperplane arrangement with an open convex set, see
Figure~\ref{fig:COMrealisable}. Examples of COMs
neither contained in the class of OMs nor in ample classes include linear extensions
of a poset or acyclic orientations of mixed graphs, see~\cite{BaChKn}, CAT(0)
Coxeter complexes of~\cite{HaPa}, hypercellular and Pasch
graphs, see~\cite{ChKnMa}, and Affine Oriented Matroids through~\cite{Baum} and~\cite{DeKn}.
Note that none of the listed examples is contained in the classes of OMs or ample classes.
Apart from the above, COMs already lead to new results and questions in various areas such as combinatorial semigroup
theory by~\cite{Margolis}, algebraic combinatorics in relation to the Varchenko
determinant by~\cite{Hochstattler,Hochstattler2}, neural codes~\cite{ItKuRo}, poset cones, see~\cite{DoKiRe},
 as well as sweeping sequences, see~\cite{PaPh}. In
particular, relations to COMs have already been established within sample
compression, by~\cite{ChKnPh_2d,ChKnPh_CUOM,Ma_Om} and ~\cite{ChChMoWa}.
A central feature of COMs is that they can be studied via their tope graphs, see Figure~\ref{fig:COMrealisable}. 
Indeed, the characterization of their tope graphs by~\cite{KnMa}
establishes an embedding of the theory of COMs into metric graph theory, with
theoretical and algorithmic implications. Namely, tope graphs of COMs form
a subclass of the ubiquitous metric graph class of partial cubes, i.e.,
isometric subgraphs of hypercubes, with applications ranging from
interconnection networks~\cite{GrPo} and  media theory~\cite{EpFaQv}, to chemical
graph theory~\cite{Eppp}. On the other hand, tope graphs of COMs can be
recognized in polynomial time~\cite{Epp,KnMa}. The graph theoretic view has
been used in several recent publications, see~\cite{KnMa1,Ma17,Ch_Johnson} and is essential to our work. 

\begin{figure}[htb]
	\centering
	\includegraphics[width=0.8\linewidth]{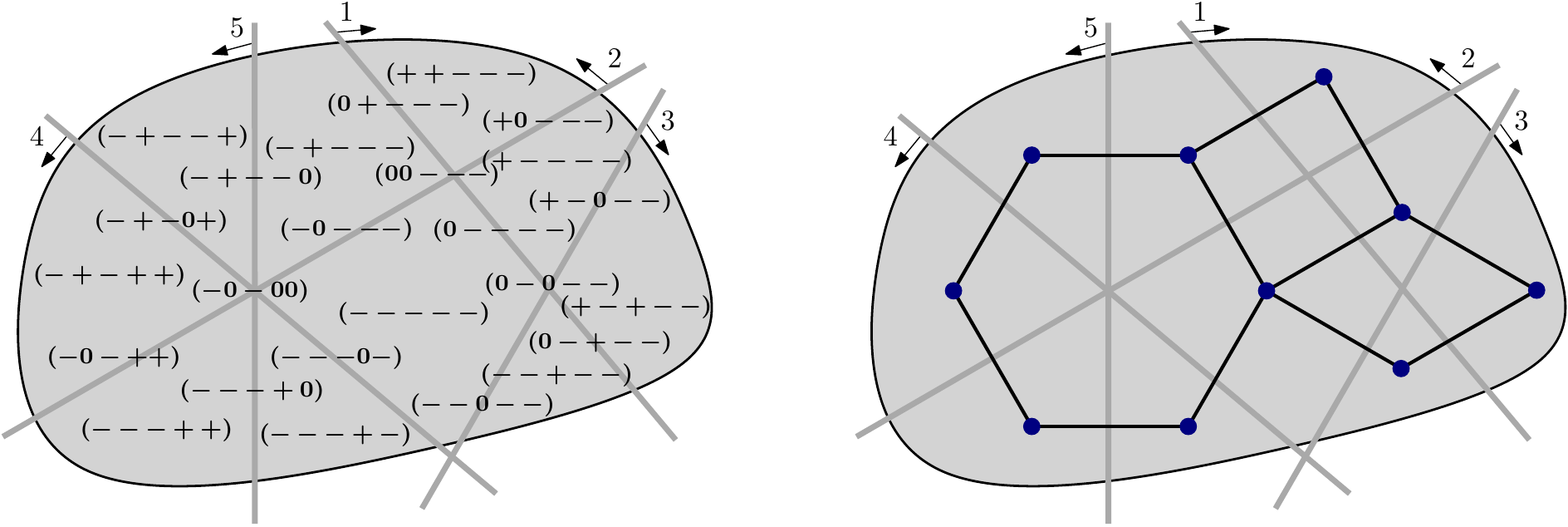}
	\caption{A realizable COM and its tope graph.
		\label{fig:COMrealisable}
	}
\end{figure}

\subsection{Labeled sample compression schemes}
As we explain later, COMs can be defined as sets of sign vectors, which is
another unifying feature for OMs and ample classes. This turns out to be
beneficial for the present paper, since the language of sign vectors is
perfectly suited for defining sample compression schemes formally. The
following formulation is due to~\cite{ChChMInRaVa}, for classical formulations,
see~\cite{LiWa,MoWa,MoYe}. Let $\U$ be a finite set, called the \emph{universe}
and $\mathcal C$ be a family of subsets of $\U$, called a \emph{concept class}
and whose elements
are called \emph{concepts}.
We  view $\mathcal C$ as a set of $\{ -1,+1\}$-vectors, i.e.,  ${\mathcal C}
\subseteq \{-1,+1\}^U$. We also consider sets  of $\{ -1,0,+1\}$-vectors, i.e.,
subsets of $\{\pm 1,0\}^U$ endowed with  the product order $\le$ between sign
vectors relative to the ordering $0 \leq -1,+1$. The sign vectors of the set
$\Sam(\mathcal C)=\bigcup_{C\in {\mathcal C}}\{ S\in \{\pm 1,0\}^U: S\leq
C\}$
are \emph{realizable samples} for $\mathcal C$.

\begin{definition} [Labeled sample compression schemes] \label{def:LSCS}
A \emph{labeled sample
compression scheme}  of size $k$ for a concept class ${\mathcal C} \subseteq
\{-1,+1\}^U$ is a pair $(\alpha,\beta)$ of mappings, where
$\alpha: \Sam(\mathcal C) \to \{\pm 1,0\}^U$ is called the  \emph{compression
function} and $\beta: \Ima (\alpha)  \to \{ -1, +1\}^U$ the
\emph{reconstruction function} such that for any realizable sample
$S\in \Sam(\mathcal C)$, it holds $\alpha(S)\le S\le \beta(\alpha(S)) \mbox{
and } |\underline{\alpha}(S)|\le k,$ where $\underline{\alpha}(S)$ is the
support of the sign vector $\alpha(S)$, i.e., the non-zero entries of $\alpha(S)$.
A labeled sample compression scheme is \emph{proper} if
$\beta(\alpha(S))\in\mathcal{C}$ for all $S\in \Sam(\mathcal{C})$.
\end{definition}

The condition $S\le \beta(\alpha(S))$
means that the restriction of $\beta(\alpha(S))$ on the support of $S$
coincides with the input sample $S$. In particular, if $S$ is a concept of
$\mathcal{C}$, then $\beta(\alpha(S))=S$, i.e., the reconstructor must
reconstruct the input concept.
 Notice
that the
labeled compression schemes of size $O(2^d)$ of~\cite{MoYe} are not proper
(i.e., $\beta(\alpha(S))$ is not necessarily a concept of ${\mathcal C}$)
and they use additional information. The compression schemes developed in
\cite{ChChMInRaVa} for balls in graphs are proper but also use additional
information. The \emph{unlabeled sample compression schemes}~\cite{KuWa}
(which are not the subject of this paper) are defined analogously, with the
difference that in the unlabeled case $\alpha(S)$ is a subset of size at most
$k$ of the support of $S$.

The definition of labeled compression scheme implies that if ${\mathcal C}'\subseteq {\mathcal C}$ and $(\alpha,\beta)$ is a
labeled sample compression scheme for $\mathcal C$, then $(\alpha,\beta)$ is a
labeled sample compression scheme for ${\mathcal C}'$. However,
$(\alpha,\beta)$ is in general not proper for  ${\mathcal C}'$.
Still, this yields an approach (suggested in~\cite{RuRuBa} and implicit in
\cite{FlWa}) to obtain improper schemes. For instance, using the result of~\cite{MoWa}
that ample classes of VC-dimension $d$ admit
labeled sample compression schemes of size $d$, one can try to extend a
given set system to an ample class without increasing the VC-dimension too
much and then apply their result. 
In~\cite{ChKnPh_2d} it is shown that partial cubes of VC-dimension 2
can be extended to ample classes of VC-dimension 2. Furthermore,
in~\cite{ChKnPh_CUOM} it is shown that OMs and complexes of uniform oriented
matroids (CUOMs) can be extended to ample classes without increasing the
VC-dimension. Thus, in these classes there exist improper labeled sample
compression schemes whose size is the VC-dimension. On the other hand,
there exist partial cubes of VC-dimension 3 that cannot be
extended to ample classes of VC-dimension 3, see~\cite{ChKnPh_CUOM},
as well as set systems of VC-dimension 2, that cannot be extended to partial
cubes of VC-dimension 2, see~\cite{ChKnPh_2d}. In~\cite{ChKnPh_CUOM} it is
conjectured that every COM of VC-dimension $d$ can be extended to an ample
class of VC-dimension $d$. This would yield improper labeled sample
compression schemes for COMs of size $d$.

\subsection{Our result}
In this paper, we follow a different strategy to give (stronger) proper labeled sample
compression schemes of size $d$ for general COMs of VC-dimension $d$,  see Theorem~\ref{main}. More
precisely,
we show that the set systems defined by the topes of COMs satisfy
the strong form of the sample compression conjecture, i.e., COMs of
VC-dimension $d$ admit \emph{proper labeled sample compression schemes of size
$d$.}

Our work  substantially extends the result of~\cite{MoWa}  for ample concept classes,  the result of~\cite{BDLi}
for concept classes arising from arrangements of affine hyperplanes (i.e., realizable Affine Oriented Matroids), and
our results~\cite{ChKnPh_CUOM} for OMs and CUOMs. Many classes of COMs are only covered by this new result.
For example, the classes of COMs mentioned in Subsection~\ref{ss:OM-COM}
are neither ample, nor affine, nor uniform. Some of these examples are realizable and can be embedded into realizable Affine Oriented Matroid to which one can apply the result of~\cite{BDLi}. However, this will lead only improper compression schemes. One important class of COMs,  which is neither realizable, nor ample,
nor affine, nor uniform, is the class of non-realizable OMs. By the  \emph{Topological Representation Theorem of Oriented
Matroids} of Folkman and Lawrence~\cite{FoLa}, the topes of OMs can be characterized as the inclusion maximal cells
of an arrangement of pseudohyperplanes. An OM is non-realizable if it it represented by a non-stretchable arrangement, i.e., an arrangement whose pseudohyperplanes cannot be replaced by linear hyperplanes.
%

To illustrate the representation by pseudohyperplanes, in Figure~\ref{fig:AOM} we give an example of an arrangement $U$ of pseudolines in ${\mathbb R}^2$ and its graph
of regions, i.e., the tope graph of the resulting COM. While this example is stretchable, there are many non-stretchable arrangements. Indeed, most OMs are non-realizable~\cite[Theorems 7.4.2 and 8.7.5]{BjLVStWhZi}. Deciding stretchability of a pseudoline arrangement and more generally realizability of an OM is a complete problem of the existential theory of the reals, hence in particular NP-hard, see~\cite{Sh}. By a result of  Edmonds and Mandel~\cite{ed-ma-82}, all arrangements of pseudohyperplanes can be considered piecewise-linear.

\begin{figure}[htb]
\centering
\includegraphics[width=0.80\textwidth]{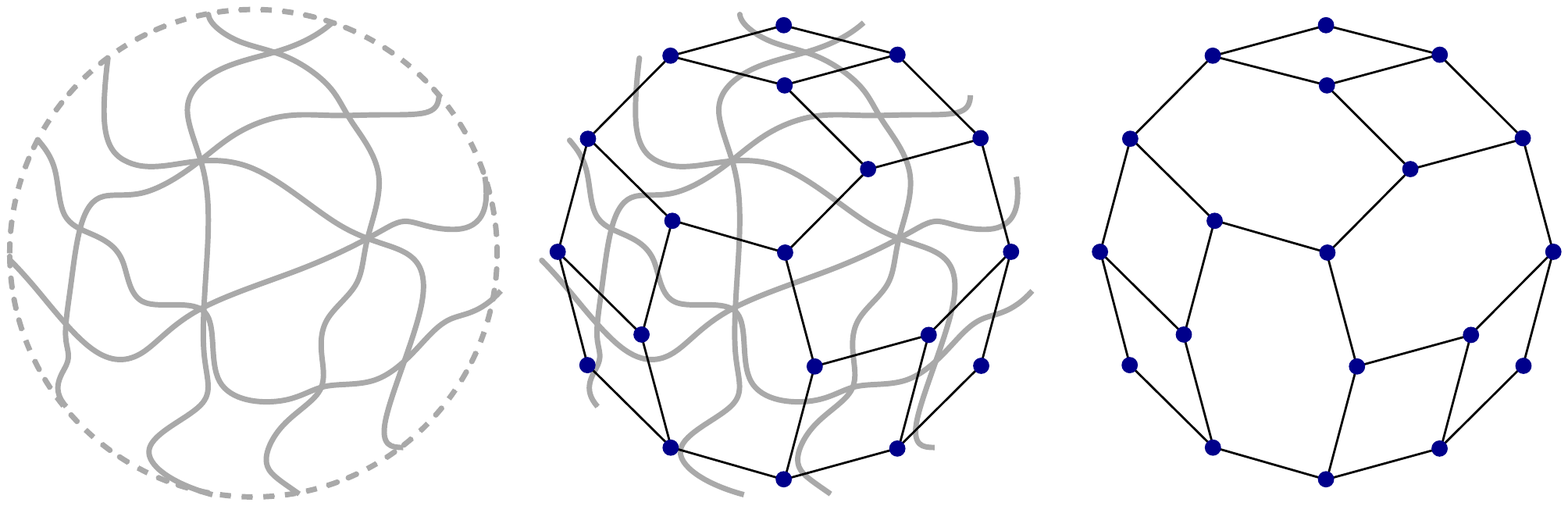}
\caption{A pseudoline arrangement $U$ and its region graph.}
\label{fig:AOM}
\end{figure}

\subsection{Pseudohyperplane arrangements and Machine Learning}  Pseudohyperplane arrangements have already arisen in the context of sample compression schemes and VC-dimension in~\cite{GaWe,Mo,RuRu_COLT,RuRu_JMLR} in the treatment of maximum and ample classes.  More recently, particular piecewise-linear pseudohyperplane arrangements and their regions occurred in the study of
deep feedforward
neural networks with ReLU activations~\cite{DeVoHaPe,GrLi,HaRo_ICML,HaRo_NIPS,MoPaChBe}.  In this theory they appear under the names ``arrangements of bent hyperplanes'' and ``activation regions'', respectively. Recall that a (trained)
feedforward neural network used to answer Yes/No (i.e., $\{ -1,+1\}$) classification problems is a particular type of function $F: {\mathbb R}^d\rightarrow {\mathbb R}$. The inputs to $F$ are data feature vectors
and the outputs are used to answer the binary classification problem by partitioning the input space ${\mathbb R}^d$ into activation regions.

Next,  we closely follow~\cite{GrLi} and~\cite{HaRo_NIPS}. A ReLU function $\ReLU: {\mathbb R}\rightarrow {\mathbb R}$ is defined by $\ReLU(x)=\max\{ 0,x\}$. ReLU is among the most popular activation functions for deep neural networks.
Let $\sigma: {\mathbb R}^d\rightarrow {\mathbb R}^d$ denote the function that applies ReLU to each coordinate.
Let $n_0,\ldots,n_k,n_{k+1}=1$ be a sequence of natural numbers and let $A_i:{\mathbb R}^{n_{i-1}}\rightarrow {\mathbb R}^{n_i}, i=1,\ldots,k+1$ be (parametrized) affine maps. A ReLU (Rectified Linear Unit) network $\mathcal N$ of
architecture $(n_0,\ldots,n_k)$, depth $k+1$, and  $n:=\sum_{i=0}^m n_i$ neurons is a neural network in which the map $F$ is defined as the composition of the layer maps $F_1=\sigma\circ A_1, \ldots, F_k=\sigma\circ A_k, F_{k+1}=A_{k+1}$.
An activation pattern for $\mathcal N$ is an assignment of a $\{ -1,+1\}$-sign to each neuron.  Given a  vector $\theta$ of trainable parameters,  the activation pattern of the neurons
defines a partition of the input space ${\mathbb R}^d$ into activation regions. The activation regions  can be viewed as the regions defined by the arrangement of bent hyperplanes associated to layers; for the precise definition see ~\cite[Section 6]{GrLi} and~\cite{HaRo_NIPS}. Activation regions are convex polyhedra~\cite{HaRo_NIPS} and one of important questions in the complexity
analysis of deep ReLU networks is counting the number of such activation regions~\cite{HaRo_ICML,HaRo_NIPS,MoPaChBe}. Notice that the arrangements of bent hyperplanes may not be arrangements of pseudohyperplanes in the classical sense~\cite{BjLVStWhZi} because
two bent hyperplanes may not intersect transversally. Transversality of arrangements of  bent hyperplanes was investigated in depth in the recent paper~\cite{GrLi}. It will be interesting to further investigate how sample compression
schemes can be useful in the setting of deep ReLU networks.

\section{Preliminaries}

\subsection{OMs and COMs}

We recall the basic 
theory OMs
and COMs from~\cite{BjLVStWhZi} and~\cite{BaChKn}, respectively.
Let $U$ be a set of size $m$ and let $\covectors$ be a {\it system of
sign vectors}, i.e., maps from $U$ to $\{-1,0,+1\}$. The elements of
$\covectors$ are referred to as \emph{covectors} and denoted by capital letters
$X, Y, Z$. For $X \in \covectors$, the subset $\uX= \{e\in U:
X_e\neq 0\}$ is the \emph{support} of $X$ and  its complement
$X^0=U\setminus \uX=\{e\in U: X_e=0\}$ is the \emph{zero set} of $X$.
For $X,Y\in \covectors$, $\Sep(X,Y)=\{e\in U: X_eY_e=-1\}$ is the
\emph{separator} of $X$ and $Y$. The \emph{composition} of $X$ and $Y$ is the
sign vector $X\circ Y$, where for all $e\in U$,
$(X\circ Y)_e = X_e$ if $X_e\ne 0$  and  $(X\circ Y)_e=Y_e$ if $X_e=0$.
Let $\leq$ be the product ordering on $\{\pm 1,0\}^{U} $
relative to the ordering $0 \leq -1, +1$.
A system of sign vectors $(U,\covectors)$ is \emph{simple}
if for each $e \in U$,
$\{X_e: X\in \covectors\}=\{-1,0,+1\}$ and for all $e\neq f$ 
there exist $X,Y \in \covectors$ with $\{X_eX_f,Y_eY_f\}=\{+1, -1\}$.
In this paper, we  consider only simple systems of sign vectors.

\begin{definition}  [OMs] \label{def:OM}
	An \emph{oriented matroid} (OM) is a system of sign vectors
	$\M=(U,\covectors)$ satisfying
	\begin{itemize}
        \item[{\bf (Z)}]  the zero sign vector ${\bf 0}$ belongs to $\covectors$.
		\item [{\bf (C)}] ({\sf Composition)} $X\circ Y \in \covectors$ for all
		$X,Y \in \covectors$.
		\item[{\bf (SE)}] ({\sf Strong elimination}) for each pair
		$X,Y\in\covectors$ and for each $e\in \Sep(X,Y)$, there exists $Z \in
		\covectors$ such that $Z_e=0$ and $Z_f=(X\circ Y)_f$ for all $f\in
		U\setminus \Sep(X,Y)$.
		\item[{\bf (Sym)}] ({\sf Symmetry}) $-\covectors=\{ -X: X\in
		\covectors\}=\covectors,$ that is, $\covectors$ is closed under sign
		reversal.
	\end{itemize}
\end{definition}
Notice that the axiom (Z) is implied by the three other axioms.   
%
The poset $(\covectors,\le)$ of an OM $\M$ with an artificial global maximum
$\hat{1}$ forms the (graded)
\emph{big face lattice} $\mathcal{F}_{\mathrm{big}}(\M)$. The length of
maximal chains of $\mathcal{F}_{\mathrm{big}}(\M)$ minus $1$
is the \emph{rank} of $\covectors$ and denoted $\rk(\M)$. The rank of the
underlying matroid $\underline{\M}$ equals $\rk(\M)$~\cite[Thm 4.1.14]{BjLVStWhZi}.
The \emph{topes} $\topes$ of $\M$ are the co-atoms of
$\mathcal{F}_{\mathrm{big}}(\M)$. 
By simplicity the topes 
are  $\{-1,+1\}$-vectors and  ${\mathcal T}$ can be seen as a family of subsets
of $U$.
For each $T\in\mathcal{T}$, an element
$e\in U$ belongs to the corresponding
set if and only if $T_e=+1$. 
The \emph{tope graph} $G(\M)$ of an OM $\M$ is the 1-inclusion
graph of the set $\mathcal T$ of topes of $\covectors$, i.e.,
the subgraph of the hypercube $\Q(U)$ induced by the vertices corresponding to
$\topes$, see Figure~\ref{fig:COMrealisable}.

In \emph{realizable OMs} (i.e., OMs arising from central hyperplane arrangements
of ${\mathbb R}^d$), $X\le Y$ for two covectors $X,Y$  if and only if the (open) cell corresponding to $X$ is contained in the cell corresponding to $Y$. Consequently, the
topes of realizable OMs are the covectors of the inclusion maximal (open) cells
(which all have dimension $d$), called \emph{regions}. Therefore, the
tope graph of a realizable OM can be viewed as the adjacency graph of regions:
the vertices of
this graph are the regions of a hyperplane arrangement and two regions are
adjacent in this graph if they are separated by a unique hyperplane of the
arrangement. The \emph{Topological Representation Theorem of Oriented
Matroids} of
\cite{FoLa}, generalizes this correspondence to all OMs:
tope graphs of OMs can be characterized as the adjacency graphs of maximal (open) cells of pseudohyperplane
arrangements in $\mathbb{R}^{d}$~\cite{BjLVStWhZi}, where $d$ is the rank of
the OM. More precisely, two topes are adjacent if and only if the corresponding
regions are separated by a unique pseudohyperplane, see Figure~\ref{fig:COMrealisable}.
It is also well-known (see for example~\cite{BjLVStWhZi}) that $\covectors$ can be recovered from its
tope graph  $G(\covectors)$ (up to isomorphism). Therefore, \emph{we can define
all terms in the language of tope graphs.} 


%

Another important axiomatization of OMs is in terms of
\emph{cocircuits} of $\covectors$. These are the atoms of
$\mathcal{F}_{\mathrm{big}}(\covectors)$. Their collection is denoted by
$\cocircuits$ and axiomatized
as follows:
a system of sign vectors  $(U,\cocircuits)$ is an \emph{oriented
matroid} (OM) if  $\cocircuits$ satisfies (Sym) and the two axioms: 
\begin{itemize}
	\item [{\bf (Inc)}] ({\sf Incomparability)}
	$\uX\subseteq\uY$
	implies $X=\pm Y$ for all $X,Y \in  \cocircuits$.
	\item[{\bf (E)}] ({\sf Elimination}) for each pair $X,Y\in\cocircuits$ with
	$X\neq -Y$ and for each $e\in \Sep(X,Y)$, there exists $Z \in  \cocircuits$
	such that $Z_e=0$  and $Z_f\in\{0,X_f,Y_f\}$ for all $f\in U$.
\end{itemize}
The set $\covectors$ of covectors can be derived from $\cocircuits$ by taking
the closure of $\cocircuits$ under composition. 

\medskip
COMs are defined by replacing the global axiom (Sym) with a weaker local axiom:

\begin{definition} [COMs] \label{def:COM}
	A \emph{complex of oriented matroids} (COM) is a system of sign vectors
	$\M=(U,\covectors)$ satisfying (SE)  and the following axiom:
	\begin{itemize}
		\item[{\bf (FS)}] ({\sf Face symmetry}) $X\circ -Y \in  \covectors$
		for all $X,Y \in  \covectors$.
	\end{itemize}
\end{definition}

One can  see that 
OMs are exactly the COMs containing the zero vector ${\bf 0}$ (axiom (Z)),
see~\cite{BaChKn}.
The twist between (Sym) and (FS) allows to keep on using the same concepts,
such as topes, tope graphs, the sign-order and the big face (semi)lattice in a
completely analogous way. On the other hand, it leads to a  combinatorial and
geometric structure that is built from OMs as cells but is much richer than
OMs. Let
$\M=(U,\covectors)$ be a COM and $X\in\covectors$ a covector. The
\emph{face} of $X$ is $\face(X):=\{Y\in\covectors: X\leq Y\}$ 
(see~\cite{BaChKn,BjLVStWhZi}) and $\cube(X)$ denotes the smallest cube of $\{-1,+1\}^U$ containing the topes of $\face(X)$.
A \emph{facet} of $\M$ is an inclusion maximal proper face. From the definition, any face $\face(X)$  consists of
the sign vectors of all faces of the subcube of $[-1,+1]^{U}$ with barycenter $X$.
By~\cite[Lemma 4]{BaChKn}, each face $\face(X)$ of a COM $\M$ is isomorphic to an OM, which however is not simple,
because all $Y\in \face(X)$ coincide on $\uX$. Thus, we consider its \emph{simplification} $\M(X)$ obtained by deleting
all the elements of $\uX$. Deletion again gives an OM as is explained in Section~\ref{subsec:del}. \emph{Ample classes}
(called also lopsided~\cite{BaChDrKo,La} or extremal
\cite{BoRa,MoWa}) are exactly the COMs, in which all faces are cubes.
%
Since OMs are COMs, each face of an OM is an OM and the facets correspond to
cocircuits.
Furthermore, by~\cite[Section 11]{BaChKn} replacing each combinatorial face
$\face(X)$ of  $\M$  by a PL-ball, we obtain a contractible cell
complex associated to each COM.
The \emph{topes} $\mathcal{T}$ and the \emph{tope graph} $G(\M)$ of a COM $\M$ are
defined as for OMs. Again, the COM $\M$ can
be recovered from $G(\M)$, see~\cite{BaChKn,KnMa}. For  $X\in \covectors$,
the topes in $\face(X)$ induce a subgraph  of $G(\M)$, which we denote by $[X]$.
We show that $[X]$ is isomorphic to the tope graph $G(\M(X))$ of $\M(X)$
and it is crucial for this paper.

\subsection{Realizable COMs} In this subsection, we recall the geometric illustration
of the axioms in the case of realizable COMs given in the paper~\cite{BaChKn}. 
Let  $U$ be an
affine arrangement of hyperplanes of ${\mathbb R}^d$ and $C$ an open convex set. Restrict the arrangement pattern to $C$,
that is, remove all sign vectors which represent the open regions disjoint from $C$. Denote the resulting set of sign vectors
by $\covectors(U,C)$ and call it a {\it realizable COM}. If $U$ is a central arrangement  with $C$ being any open convex set containing the origin,
then $\covectors(U,C)$ coincides with the realizable oriented matroid of $U$. If the arrangement $U$ is affine and $C$ is the entire space, then
$\covectors(U,C)$ coincides with the realizable affine oriented matroid of $U$.
The realizable ample sets arise by taking  the central arrangement $U$ of all coordinate hyperplanes restricted to
an arbitrary open convex set $C$ of ${\mathbb R}^d$ (this model was first considered in~\cite{La}).

We argue, why a realizable COM satisfies the axioms from Definition~\ref{def:COM}.
Let $X$ and $Y$ be sign vectors belonging to $\covectors(U,C)$ and designating  two open regions of $C$ defined by $U$. Let
$x,y$ be two points in these regions. Connect $x,y$ by a line segment and choose
$\epsilon> 0$ so that the open ball of radius $\epsilon$ around $x$ is contained in $C$ and
intersects only those hyperplanes from $U$ containing $x$. Pick any point $w$ from the intersection of this ball with
the open line segment between $x$ and $y$. The corresponding sign vector $W$ is the composition $X\circ Y$, establishing (C).
If we select a point $v$ on the ray from $y$ through  $x$ within the $\epsilon$-ball but beyond $x$, then the
corresponding sign vector $V$ has the opposite signs as $W$ at the coordinates corresponding to the hyperplanes
from $U$ containing $x$ and not including the ray from $y$ through $x$. Hence, $V=X\circ -Y$, yielding (FS).
Now, assume that the hyperplane $e$ from $U$ separates $x$ and $y$, that is, the line segment between $x$ and $y$ crosses $e$ at
some point $z$. The corresponding sign vector $Z$ is then zero at $e$ and equals the composition $X\circ Y$ at all coordinates
where $X$ and $Y$ are sign-consistent, establishing (SE). If the hyperplanes of $U$ have a non-empty intersection in $C$, then
any point $o$ from this intersection corresponds to the zero sign vector, showing that central hyperplane arrangements define OMs.
In this case, $\covectors(U,C)$ coincides with $\covectors(U,{\mathbb R}^d)$ as well as with  $\covectors(U,C_{\epsilon})$,
where $C_{\epsilon}$ is any  open ball centered at $o$.
The face $\face(X)$ of a covector $X\in \covectors(U,C)$ is obtained by taking any point $x\in C$ corresponding to $X$ and a small
$\epsilon$-ball $C_{\epsilon}$ centered at $x$. Then $\face(X)$ coincides with the OM $\covectors(U,C_{\epsilon})$.
Finally, notice that the topes of $\covectors(U,C)$ correspond to the connected components of $C$ minus the hyperplanes of $U$.
Two such topes are adjacent in the tope graph if and only if the corresponding regions are separated by a single hyperplane.
Furthermore, the distance between any two topes in the tope graph of  $\covectors(U,C)$  is equal to the number of hyperplanes separating
the two regions corresponding to these topes (for $C={\mathbb R}^d$ this was proved by  Deligne~\cite[Proposition 1.3]{De}).

\subsection{Deletions and duality}\label{subsec:del}

We continue with deletions in OMs and COMs. 
Let $\M=(U,\covectors)$ be a COM and $A\subseteq U$. Given a sign vector
$X\in\{\pm 1, 0\}^U$,  by $X\setminus A$ (or by $X_{|U\setminus A}$) we refer to
the \emph{restriction}
of $X$ to $U\setminus A$, that is $X\setminus A\in\{\pm1, 0\}^{U\setminus A}$
with $(X\setminus A)_e=X_e$ for all $e\in U\setminus A$.
The \emph{deletion} of $A$ is defined as $\M\setminus A=(U\setminus
A,\covectors\setminus A)$, where $\covectors\setminus A:=\{X\setminus A:
X\in\covectors\}$.
We often consider the following type of deletion. For a covector $X\in
\covectors$, we denote by $\M(X)=(U\setminus \uX, \face(X)\setminus
\uX)$
the simple OM defined by the face $\face(X)$. Note that $\M(X)=\M\setminus \uX$, since for every $Y\in \covectors$ we have that $Y\setminus\uX=(X\circ Y)\setminus\uX$ and $X\circ Y\in \face(X)$.
The classes of COMs and OMs 
are closed under deletion, see~\cite[Lemma 1]{BaChKn}.
The cocircuits and the covectors of deletions of OMs are described in the following way:


%


\begin{lemma}\label{lem:deletecocircuits}~\cite{BjLVStWhZi}
 Let $\M=(U,\covectors)$ be an $\OM$ with the set of  cocircuits $\cocircuits$
 and $A\subseteq U$. Then the cocircuits of $\M\setminus A$ are
the minimal elements of $\cocircuits\setminus A$ and the covectors of $\M\setminus A$ are
$\covectors\setminus A$.
\end{lemma}

%

We briefly recall the duality  of OMs, see~\cite[Section 3.4]{BjLVStWhZi}. The
duality is defined via orthogonality of
circuits and cocircuits, which can be viewed as a synthetic version of
classical orthogonality of vectors.
Two sign-vectors $X,Y\in\{\pm 1,0\}^U$ are \emph{orthogonal}, denoted $X\bot
Y$, if either $\uX\cap\uY=\varnothing$ or there are
$e,f\in \uX\cap\uY$ such that $X_eY_e=-X_fY_f$. 
Oriented matroids can be defined in terms of their \emph{vectors} $\mathcal V$ and
\emph{circuits} $\mathcal C$, which can be derived from the cocircuits
$\circuits^*$ using the following result:

\begin{theorem}\label{lem:orthogonality}\cite[Theorem 3.4.3 and Proposition
3.7.12]{BjLVStWhZi}
	Let $\M$ be an OM. The set $\mathcal V$   consists of all $Y\in \{ \pm 1,0\}^U$ such that $Y\bot X$ for any $X\in \circuits^*$ and
	$\mathcal C$ consists of the minimal members of $\mathcal V \setminus
	\{\textbf{0}\}$. 
\end{theorem}



We will also make use of the version of Lemma~\ref{lem:deletecocircuits} for circuits:

\begin{lemma}\label{lem:deletecircuits}~\cite{BjLVStWhZi} 
 Let $\M=(U,\covectors)$ be an $\OM$ with the set of  circuits $\circuits$
 and $A\subseteq U$. Then the circuits of $\M\setminus A$ are
$X\in \circuits$ such that $\uX\cap A=\varnothing$.
\end{lemma}

\begin{remark}
 Throughout the paper we will use letters like $S, S'$ for samples, $T, T'$ for topes, and $X,Y,Z$ for cocircuits, covectors, and circuits.
\end{remark}

\subsection{Partial cubes and pc-minors} It is well-known, see for example~\cite{BjLVStWhZi,BaChKn}, that tope graphs of OMs and COMs are
partial cubes, which we introduce now. 
%
Let $G=(V,E)$ be a finite, connected,
 simple graph. The {\it distance} $d(u,v):=d_G(u,v)$ between vertices $u$ and
$v$ is the length of a shortest $(u,v)$-path, and the {\it interval}
$I(u,v):=\{ x\in V: d(u,x)+d(x,v)=d(u,v)\}$
consists of all vertices on shortest $(u,v)$-paths. A subgraph $H$
 is {\it convex} if $I(u,v)\subseteq H$ for any $u,v\in H$ and {\it
 gated}~\cite{DrSch} if for every vertex $x\notin H$ there
exists a vertex $x'$ (the {\it gate} of $x$) in $H$ such that $x'\in I(x,y)$
for  each vertex $y$
of $H$. 
It is easy to see that gates are unique 
and that gated sets are convex. An induced subgraph $H$ of $G$ is
{\it isometric} if the distance between vertices in $H$ is
the same as that in $G$.
A graph $G=(V,E)$ is {\it isometrically embeddable} into a graph $H=(W,F)$ if
there exists $\varphi : V\rightarrow W$ such that $d_H(\varphi
(u),\varphi (v))=d_G(u,v)$ for all $u,v\in V$. 
A graph $G$ is a {\it partial cube} if it
admits an isometric embedding into a hypercube $\Q_m=\Q(U)$. For an edge $uv$ of
$G$, let $W(u,v)=\{ x\in V: d(x,u)<d(x,v)\}$. For an edge $uv$,
the sets $W(u,v)$ and $W(v,u)$ are called {\it complementary halfspaces} of $G$.

\medskip
\begin{theorem}~\cite{Dj} \label{Djokovic}
	A graph $G$ is a partial cube if and only if $G$ is bipartite and for any
	edge $uv$ the sets $W(u,v)$ and $W(v,u)$ are convex.
\end{theorem}

Djokovi\'{c}~\cite{Dj} introduced the following binary relation $\Theta$  on
the edges of $G$:  for two edges $e=uv$ and $e'=u'v'$, we set $e\Theta e'$ if
$u'\in W(u,v)$ and $v'\in W(v,u)$. If $G$ is a partial cube, then
 $\Theta$ is an equivalence relation. Each $\Theta$-class $E_e$ corresponds to a coordinate $e\in U$ of the hypercube $\Q(U)$
 into which $G$ is isometrically embedded. 
Let $\{ G^-_e,G^+_e\}$ be the
complementary halfspaces of $G$ defined by setting $G^-_e:=G(W(u,v))$
and $G^+_e:=G(W(v,u))$ for an arbitrary edge $uv\in E_e$ (for $S\subseteq V(G)$ we denote by $G(S)$
the subgraph of $G$ induced by $S$).  
An {\it elementary pc-restriction} consists of taking one of the
halfspaces $G^-_e$ and $G^+_e$. A {\it pc-restriction} is a convex subgraph of $G$
induced by the
intersection of a set of halfspaces of $G$. 
Since any convex subgraph of a partial cube $G$ is the intersection of
halfspaces~\cite{AlKn,Ba,Ch_thesis}, the pc-restrictions of
$G$ coincide with the convex subgraphs of $G$. Denote by $\pi_e(G)$ an
\emph{elementary pc-contraction}, i.e., the graph obtained from
$G$ by contracting the edges in $E_e$. For a vertex $v$ of $G$, let $\pi_e(v)$
be
the image of $v$ under the contraction. We apply $\pi_e$ to subsets $S\subseteq
V$, by setting $\pi_e(S):=\{\pi_e(v): v\in S\}$.
By~\cite[Theorem 3]{Ch_hamming}, 
the class of partial cubes is closed under pc-contractions. Since pc-contractions
commute,  
for a set $A$ of $\Theta$-classes, we denote
by $\pi_A(G)$ the isometric subgraph of $\Q(U\setminus A)$ obtained from $G$ by
contracting the equivalence classes of edges from $A$.
pc-Contractions and pc-restrictions also commute in partial cubes. A {\it pc-minor}
of $G$ is a
partial cube obtained from $G$ by pc-restrictions and pc-contractions. 
A deletion $\M\setminus A$ in a COM $\M$  translates to the contraction of the $\Theta$-classes
$E_e$ with $e\in A$  in its tope
graph $G(\M)$.   Since tope graphs of COMs and OMs are partial cubes, we can  describe pc-restrictions
and pc-contractions 
on sign-vectors in terms of partial cubes. 
First recall the following fundamental lemma from~\cite{BaChKn} and ~\cite{KnMa}:

\begin{lemma} \label{lem:face-gated}
For each covector $X$ of a COM $\M$, 
$[X]$ 
is a gated subgraph of the tope graph $G(\M)$ of $\M$. Moreover, for any tope $T$ of $\M$,  $X\circ T$  is the gate of $T$  in $[X]$
and in $\cube(X)$.
\end{lemma}

Let $G$ be an isometric subgraph of the hypercube $Q(U)$ and $H$ be an isometric subgraph of the hypercube $Q(U\setminus A)$ for some $A\subseteq U$. We say that $G$ and $H$ are
\emph{$U$-isomorphic} if there exists an isomorphism between $G$ and $H$ which maps each edge of a $\Theta$-class $E_e$ of $G$ to an edge of $E_e$ of $H$.

\begin{lemma} \label{minor-pc-minor}  Let $\M=(U,\covectors)$ be a COM and
$A\subseteq U$. Then $\pi_A(G(\M))$ is the tope graph of $\M\setminus A$. If $X\in \covectors$,
then the tope graph $[X]$ of $(U,F(X))$ is $U$-isomorphic to the tope graph $G(\M(X))=\pi_{\uX}(G(\M))$
of $\M(X)=\M\setminus \uX$. 
\end{lemma}

\begin{proof} That $G(\M\setminus A)=\pi_A(G(\M))$ follows from the equivalence between deletion in COMs and pc-contraction  in their tope graphs.
To prove that $[X]$ is $U$-isomorphic to $G(\M(X))$, note that $[X]$ is obtained from $G(\M)$ by a pc-restriction: $[X]$ is the intersection of the halfspaces defined by the $\Theta$-classes $E_e$ with $e\in \uX$ and containing $[X]$. We assert that the pc-restrictions and the
pc-contractions over  $\uX$ give the same result, i.e., that $\pi_{\uX}(G(\M))$ is $U$-isomorphic to $[X]$. Indeed, by Lemma~\ref{lem:face-gated},  $[X]$ is a gated subgraph of $G(\M)$. Pick any $e\in \uX$ and consider the
elementary pc-contraction of the class $E_e$. By Lemma~\ref{lem:face-gated}, the gate of any tope $T$ of $\M$ in $[X]$ and in the cube $\cube(X)$ is $X\circ T$. Therefore, if $T,T'\in \{-1,+1\}^U$ such that $\Sep(T,T')=e$, $T$ is a vertex of $G(\M)$
not belonging to $[X]$, and $T'$ belongs to $\cube(X)$, then necessarily  $T'=X\circ T$ and thus $T'$ must be a vertex of $[X]$. This implies that the intersection of the cube $\cube(X)$ with $\pi_e(G(\M))$
(which is the tope graph of the face of $X$ in $\M\setminus e$) coincides with $[X]$. Consequently, $[X]$ coincides with $\pi_{e}(G(\M))$. Performing elementary pc-contractions for all elements of $\uX$
we conclude that $[X]$ is $U$-isomorphic to $\pi_{\uX}(G(\M))=G(\M(X))$.
\end{proof}




\subsection{VC-dimension}

Let $\SS$ be a family of subsets of an $m$-element set $U$.  A subset $X$ of
$U$ is \emph{shattered} by $\SS$ if for
all $Y\subseteq X$ there exists $S\in\SS$ such that $S\cap X=Y$. The
\emph{Vapnik-Chervonenkis dimension}
(VC-dimension)~~\cite{VaCh} $\vcd(\SS)$ of $\SS$ is the cardinality of the
largest subset of $U$ shattered by $\SS$.
Any set system $\SS\subseteq 2^U$ can be viewed as a subset of
vertices of the $m$-dimensional hypercube $\Q_m=\Q(U)$. Denote by $G(\SS)$ the
\emph{1-inclusion
graph} of $\SS$, i.e., the subgraph of $\Q(U)$ induced
by the vertices of $\Q(U)$ corresponding to $\SS$. 
A subgraph $G$ of $\Q(U)$ has VC-dimension $d$ if $G$ is the 1-inclusion graph
of a set system of VC-dimension $d$.
For partial cubes, the notions of
shattering and VC-dimension can be formulated in terms of pc-minors. First, note that
if $G'$ is a pc-minor of a partial cube $G$ and $G'$ shatters a subset $X$ of $U$,
then $G$ also shatters $X$. Thus a partial cube $G$ has VC-dimension $\le d$ if
and only if $G$ does not have the hypercube $\Q_{d+1}$ as a pc-minor. More precisely
a subset $D\subseteq U$ of the $\Theta$-classes of $G$ shatters $G$ if
$\pi_{U\setminus D}(G)$ is isomorphic to a hypercube. This is well-defined,
since the embeddings of partial cubes 
are unique up to isomorphism, see e.g.~\cite[Chapter 5]{Ov1}.

The \emph{VC-dimension} $\vcd(\M)$ of a COM $\M=(U,\covectors)$ is the
VC-dimension of its tope graph $G(\M)$ and we say that $D\subseteq U$ is shattered by $\M$ if $D$ is shattered by $G(\M)$.
The \emph{VC-dimension} $\vcd(X)$ of a covector $X\in \covectors$ of $\M$ is
the VC-dimension of the OM $\M(X)$, i.e., by Lemma \ref{minor-pc-minor}, it is the VC-dimension of the graph
$[X]$. The VC-dimension of OMs, COMs, and their covectors
can be expressed in the following way:

\begin{lemma} \label{lem:rankfunction}~\cite[Lemma 13]{ChKnPh_CUOM} For a COM
$\M$, $\vcd(\M)=\max\{ \vcd(\M(X)): X\in \covectors\}$.
If $\M$ is an OM and $X$ a cocircuit of $\M$, then $\vcd(X)+1=\vcd(\M)=\rk(\M)$.
\end{lemma}

That $\vcd(X)=\vcd(\M)-1$  for cocircuits $X$ of an OM $\M$ follows from  the
fact that
the cocircuits  are atoms of the big face lattice
$\mathcal{F}_{\mathrm{big}}(\M)$ and this lattice   is graded.


%
%
%

\section{Auxiliary results}\label{sec:prel-results}
We establish and recall several auxiliary results about OMs and COMs. We also
develop a correspondence between realizable samples and convex subgraphs of
partial cubes. Finally, we define upper and lower covectors for a given sample, which are crucial notions for the main result.


\subsection{More about shattering in  OMs and COMs}
We continue with new results about shattering in OMs and COMs. Let $G$ be a
partial cube,  $H$ a convex subgraph, and $E_e$ a $\Theta$-class of $G$. We say
that $E_e$ {\it crosses}
$H$  if $H$ contains an edge of $E_e$. If $E_e$ does not cross $H$ and there
exists an edge $uv$ of $E_e$ with $u\in H$ and $v\notin H$, then $E_e$  and $H$ {\it
osculate}. Otherwise, $E_e$ is {\it disjoint} from $H$. Denote by $\osc(H)$ the
set of all $e$ such that $E_e$ osculates
with $H$ and by $\cross(H)$ the set of all $e$ such that $E_e$ crosses $H$.

%

\begin{lemma}\label{lem:oscilate}
Let $G$ be a partial cube, $H$ a convex subgraph of $G$, and $e\notin \osc(H)$. Then
$\pi_e(H)$ is convex in $\pi_e(G)$ and $\osc(\pi_e(H))=\osc(H)$,
where $\osc(H)$ and $\osc(\pi_e(H))$ are considered in $G$ and
$\pi_e(G)$, respectively.
\end{lemma}

\begin{proof}
	Let $H'=\pi_e(H)$. First, since $e \notin \osc(H)$, the fact that
	$H'$ is a convex subgraph of $\pi_e(G)$ comes from~\cite[Lemma 5]{ChKnMa}.
	Then, the inclusion $\osc(H)\subseteq \osc(H')$ is obvious. If there exists
	$e'\in \osc(H')\setminus \osc(H)$, then there exists an edge
	$\pi_e(u)\pi_e(v)$ in $\pi_e(E_{e'})$ with $\pi_e(u)\in V(H')$ and
	$\pi_e(v)\notin V(H')$. Then $\pi_e(u)\pi_e(v)$ comes from an edge $uv$ of
	$G$ belonging to $E_{e'}$. Since $e'\notin \osc(H)$ and the vertex $v$ does not belong to $H$, the vertex $u$ also
    does not belong to $H$. This implies that there exists an edge $uw$ of
	$E_e$ with $w\in V(H)$. If $E_e$ and $H$ contain an
    edge $u'w'$ and say $d(u,u')<d(w,u')$, then  $u\in I(w,u')$, which contradicts
    the convexity of $H$. Thus $E_e$ and $H$ osculate,
    a contradiction. This establishes the equality $\osc(\pi_e(H))=\osc(H)$.
\end{proof}

\begin{lemma}\label{lem:gatedshatter}
 Let $G$ be a partial cube and $H$ a gated subgraph of $G$. If  $D\subseteq
 \cross(H)$ is  shattered by $G$, then $D$ is shattered by $H$.
 \end{lemma}

\begin{proof} Pick any $\Theta$-class $E_e$ with $e\in D$ and let $v$ be any vertex of
	$G$. If $v$ belongs to the halfspace $G_e^-$  of $G$, then the gate $v'$ of
	$v$ in $H$ also belongs to $G_e^-$. Indeed, since $E_e$ crosses $H$, there
	exists a vertex $w\in G_e^-\cap H$. Then $v'\in I(v,w)\subset G^-_e$ by
	convexity of $G^-_e$ and because  $v'$ is the gate of $v$ in $H$. Analogously,
	if $v\in G^+_e$, then $v'\in G^+_e$.
	
	Since $G$ shatters $D$, for any sign vector $X\in \{ -,+\}^D=\{ -1,+1\}^D$, there
	exists a vertex  $v_X$ of $G$, whose restriction to $D$ coincides with $X$.
	This means that for any $e\in D$, the vertex $v_X$ belongs to the halfspace
	$G_e^{X_e}$. Since the gate $v'_X$ of $v_X$ in $H$ also belongs to
	$G_e^{X_e}$, the restriction of $v'_X$ to $D$ also coincides with $X$. This
	implies that $H$ also shatters $D$.
\end{proof}





The next lemma  shows  that the sets shattered by an OM $\M$ are exactly the \emph{independent sets} of the underlying
matroid $\underline{\M}$, i.e.,
the sets not containing supports of circuits of $\M$. 

\begin{lemma}\label{lem:shatterandcircuits}
 Let $\M=(U,\covectors)$ be an $\OM$ and $D$ be a subset of $U$. Then $D$ is
 shattered by $\M$ if and only if $D$ is independent in the underlying matroid
 $\underline{\M}$. 
\end{lemma}

\begin{proof}
By definition $D$ is shattered by $\M$ if and only if $D$ is
 shattered by $G(\M)$. This is equivalent to $\pi_{U\setminus D}(G(\M))=Q_{U\setminus D}$. But since  we have $\pi_{U\setminus D}(G(\M))=G(\M_{|D})$
 this means $\covectors(\M_{|D})=\{\pm 1,0\}^D$. By Theorem~\ref{lem:orthogonality} this is equivalent to $\mathcal{V}(\M_{|D})=\{\mathbf{0}\}$ and $\mathcal{C}(\M_{|D})=\varnothing$.
 Applying Lemma~\ref{lem:deletecircuits} this just means that the support of no circuit of $\M$ is contained in $D$. By definition this means that $D$ is independent in $\underline{\M}$.
\end{proof}

An \emph{antipode} of a vertex $v$ in a partial cube $G$ is a (necessarily
unique) vertex $-v$ such that $G=I(v,-v)$.
A partial cube $G$ is \emph{antipodal} if all its vertices have antipodes. By
(Sym), a tope graph of a COM is the tope graph of an OM if and only if it is
antipodal, see~\cite{KnMa}.

The next lemma can be seen as dual
analogue of Lemma~\ref{lem:shatterandcircuits}. It shows that the VC-dimension
of OMs is defined locally at each tope $T$, by shattering subsets of $\osc(T)$.

\begin{lemma}\label{lem:localshattering}
 Let $\M=(U,\covectors)$ be an $\OM$ of rank $d$ with tope graph $G(\M)$. For
 any tope $T$ of $\M$, $\osc(T)$ contains a subset $D$ of size $d$ shattered
 by $\M$.
\end{lemma}

\begin{proof} We proceed by induction on the size of $U$.  If $\osc(T)=U$, then
	we are obviously done. Thus suppose that there exists $e\notin \osc(T)$.
	Consider the tope graph	$G'=\pi_e(G)$ of the oriented matroid
	$\M'=\M\setminus e$. Let $T'=\pi_e(T)$.	Then $\osc(T')=\osc(T)$ by Lemma
	\ref{lem:oscilate}. If $\rk(\M')=d$, by	the induction hypothesis the set
	$\osc(T')$ contains a subset $D$ of size $d$ shattered by $G'$. Since $G'$
	is a pc-minor of $G$, the set $D\subset \osc(T')=\osc(T)$ is also shattered
    by $G$ and we are done.
	
	Thus, let $\rk(\M')<\rk(\M)$. If the $\Theta$-class $E_e$ of $G$ crosses
	the faces $\face(X)$ of all	cocircuits $X\in\covectors$, then $\covectors$
	is not simple. Therefore, there exists a cocircuit $X\in\covectors$ whose
	face $\face(X)$ is not crossed by $E_e$. However, since	when we contract
	$E_e$ the rank decreases by 1, the resulting OM $\M'$ coincides with
	$\face(X)$. Indeed,	after contraction the rank of $\face(X)$ remains the
	same. Hence, if $X$ would remain a cocircuit, then the global rank would not
	decrease. Hence, $G'$ is the tope graph of $\M(X)$. Since $G$ is an
	antipodal partial cube and $G_e^+=\face(X)$, we have $G_e^-\cong G_e^+$.
	This shows that $G\cong G_e^+\product K_2\cong G'\product K_2$.	This
	implies that $E_e$ osculate with $\{ T\}$ in $G$, contrary to the
	assumption $e\notin \osc(T)$.
%
%
%
%
\end{proof}

Next we give a shattering property of COMs. 
The {\it distance} $d(A,B)$ between sets $A,B$ of vertices  of  $G$ is $\min \{
d(a,b): a\in A, b\in B\}$. 
The set $\pr_B(A)=\{ a\in A: d(a,B)=d(A,B)\}$ is the {\it metric projection}
of $B$ on $A$. 
For two covectors $X,Y\in \covectors$ of a COM $\M$,
we denote by $\pr_{[X]}([Y])$ the metric
projection of  $[X]$ on  $[Y]$ in $G(\M)$.
Since $[X]$ and $[Y]$ are gated by Lemma~\ref{lem:face-gated},
$\pr_{[X]}([Y])$ consists of the gates of vertices of $[X]$ in $[Y]$, see~\cite{DrSch}.
Two faces $\face(X)$ and $\face(Y)$ 
of $\M$ are {\it parallel} if
$\pr_{[X]}([Y])=[Y]$ and  $\pr_{[Y]}([X])=[X]$. A {\it gallery} between
two parallel faces $\face(X)$ and $\face(Y)$ of  $\M$ is a sequence of faces
$(\face(X)=\face(X_0), \face(X_1),\ldots, \face(X_{k-1}),\face(X_k)=\face(Y))$
such that either $k=0$ (i.e., $\face(X)=\face(Y)$) or any two
faces of this sequence are parallel and any two consecutive faces
$\face(X_{i-1}), \face(X_i)$ are facets of a common face of $\covectors$. A {\it
geodesic gallery} between $\face(X)$ and $\face(Y)$ is a gallery of length
$|\Sep(X,Y)|$. Two parallel faces $\face(X), \face(Y)$ are {\it
adjacent} if $|\Sep(X,Y)|=1$, i.e., $\face(X)$ and $\face(Y)$ are opposite
facets of a
face of $\covectors$. See Figure~\ref{fig:galerie} and recall the following
result:


\begin{figure}[htb]
	\centering
	\includegraphics[width=0.8\linewidth]{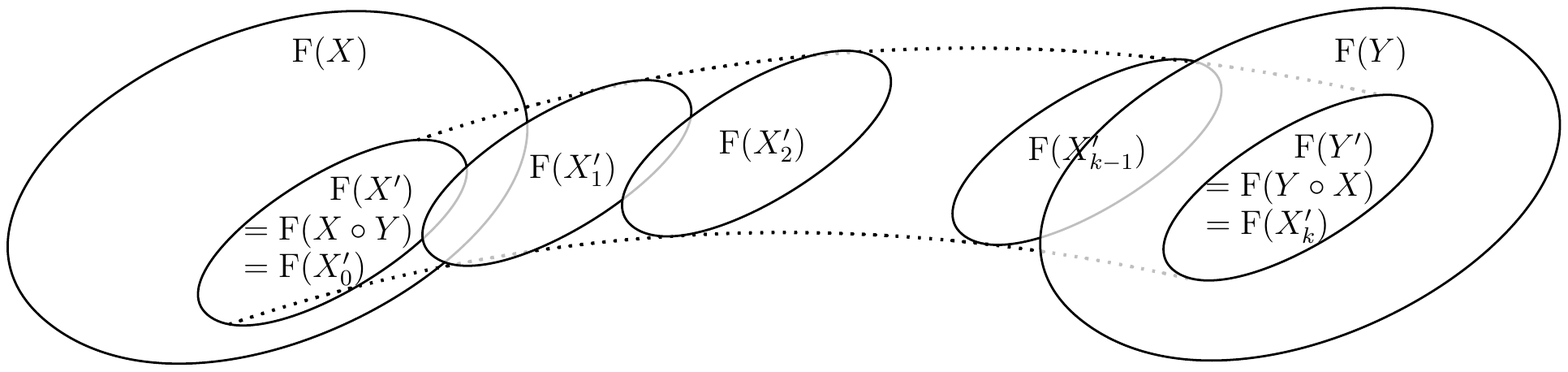}
	\caption{
		\label{fig:galerie}
		Illustration of Lemmas~\ref{prop:projection} and~\ref{lem:galleries}.
	}
\end{figure}

\begin{lemma}
	\label{prop:projection}  ~\cite[Proposition 8]{ChKnPh_CUOM}
	Let $\M=(U,\covectors)$ be a COM and $X,Y\in \covectors$ (not necessarily distinct). Then: 
	\begin{itemize}
		\item[(i)]   $d([X],[Y])=|\Sep(X,Y)|$ and the gates of $[Y]$ in $[X]$ are
		the vertices of $[X\circ Y]\subseteq  [X]$;
        \item[(ii)] $\face(X)$ and $\face(Y)$ are parallel if and only if
        $\uX=\uY$. If $\face(X)$ and $\face(Y)$  are
		parallel, then they are connected by a geodesic gallery;
		\item[(iii)] $\pr_{[Y]}([X])=[X\circ Y],$  $\pr_{[X]}([Y])=[Y\circ X]$,
		and $\face(X\circ Y)$ and
		$\face(Y\circ X)$ are parallel.
     \end{itemize}
\end{lemma}

A covector $X\in \covectors$ of a COM $\M=(U,\covectors)$ \emph{maximally
shatters} a set $D\subseteq U$
if $[X]$ shatters $D$ but $[X]$ does not shatter any superset of $D$. We also
say that  $X\in \covectors$ \emph{locally maximally shatters} a set $D$ if $[X]$
shatters $D$ but $D$ is not shattered by $[X']$ for any covector $X'>X$.

\begin{lemma}\label{lem:galleries}
Let $\M=(U,\covectors)$ be a COM and $X,Y\in \covectors$ (not necessarily distinct). Then: 
	\begin{itemize}
		\item[(i)] if $[X]$ and $[Y]$ shatter $D$, then the projections
		$[X\circ Y]$ and
		$[Y\circ X]$ also shatter $D$;
        \item[(ii)] if $X\ne \pm Y$ and $[X]$ maximally shatters $D$ and $[Y]$ shatters $D$,
        then $[X\circ Y]=[X]$ and $\face(X)$ is not a facet of $\M$;
        \item[(iii)] if both $[X]$ and $[Y]$ shatter  $D$, then there exist
        covectors $X'\geq X,  Y'\geq Y$ such that $[X']$ and $[Y']$ both
        maximally shatter $D$, and $\face(X')$ and $\face(Y')$ are parallel. In particular,
        if $[X]$ shatters $D$, then there exists a covector $X'\geq X$ such that $[X']$ maximally
        shatters $D$.
     \end{itemize}
\end{lemma}

\begin{proof}  \textbf{Property (i):} Since $[X]$ and $[Y]$ shatter $D$, for any sign
	vector
	$Z\in \{ \pm 1\}^D$ we can find two topes $T'\in [X]$ and $T''\in [Y]$,
	such that $T'_{|D}=Z=T''_{|D}$. Since $X\le T'$ and $Y\le T''$, from
	$T'_{|D}=Z=T''_{|D}$ we conclude that $(X\circ Y)_{|D}<Z$ and in $[X\circ
	Y]$ we can find a tope $T$ whose restriction to $D$ coincides with $Z$.
	This proves that $[X\circ Y]$ shatters $D$, establishing (i).
	
	\textbf{Property (ii):} If $[X]$ maximally shatters $D$, then
	$\vcd(X)=|D|=:d$. By property (i), $[X\circ Y]$ also shatters $D$. If
	$\face(X\circ Y)$ is a proper face of $\face(X)$, then we obtain a
	contradiction with Lemma~\ref{lem:rankfunction} applied to the OM $\M(X)$. Thus $\face(X\circ
	Y)=\face(X)$, showing that $X=X\circ Y$. This establishes the first
	assertion. By Lemma~\ref{prop:projection}, the faces $\face(X)$ and $\face(Y\circ X)$ are parallel
    and therefore are connected by a geodesic gallery
	$(\face(X)=\face(X_0),\face(X_1), \ldots, \face(X_k)=\face(Y\circ X))$.
	Then  either  $k=0$ and $\face(X)=\face(Y\circ X)$ holds or  $\face(X)$ and $\face(X_1)$ are facets of a common face of
	$\covectors$. In the first case, since $X\ne \pm Y$,   we conclude that $\face(X)$  is a
	proper face of $\face(Y)$, and thus is not a facet of $\M$.
	 In the second case, $\face(X)$ is not a facet of $\M$ either. This proves (ii).
	
	\textbf{Property (iii):} Let $d=|D|$. We can suppose that both $X$ and $Y$
	locally maximally shatter the set $D$. Indeed, if $D$ is shattered by a proper face
	$\face(X')$ of $\face(X)$, then we can replace the pair $X,Y$ by the pair
	$X',Y$ so that $[X']$ and $[Y]$ still shatter $D$. Thus $D$ is not
	shattered by any proper faces of $\face(X)$ and $\face(Y)$. Since by (i),
	$D$ is shattered by	$[X\circ Y]$ and $[Y\circ X]$, we conclude that
	$X=X\circ Y$ and $Y=Y\circ X$ and thus the faces $\face(X)$ and $\face(Y)$
	are parallel.
	
It remains to show that $[X]$ and $[Y]$ maximally shatter $D$. Suppose by
	way of contradiction that $[X]$ shatters a larger set $D':=D\cup \{ e\}$.
	Consider the OM $\M'=\M(X)\setminus (U\setminus D')$. Note that $\M'$ maximally shatters
	$D'$, i.e.,	$\vcd(\M')=d+1$. Since $[X]$
	shatters $D'$, the covectors of $\M'$ are $\{\pm 1,0\}^{D'}$. Let $X''$ be a cocircuit of $\M'$ with $\underline{X'}=\{e\}$.
	By Lemma~\ref{lem:rankfunction} applied to $\M'$, we conclude that $X''$
	has VC-dimension $d$. Hence, $X''$ must shatter the set $D$. By Lemma
	\ref{lem:deletecocircuits},	there is a cocircuit $X'$ of $\face(X)$ such
	that $X''=X'\setminus (U\setminus D')$. Since $X''$ shatters $D$, $X'$ also
	shatters $D$. Since $X<X'$, this contradicts our assumption that $X$
	locally maximally shatters $D$.	The second assertion follows by applying the first
    assertion with $Y=X$.  This establishes (iii).
\end{proof}

\subsection{Realizable and full samples as convex subgraphs}
Let $\M=(U,\covectors)$, where $\covectors\subset \{\pm 1, 0\}^U$ is a system of sign vectors whose topes
$\topes$ induce an isometric subgraph $G$ of $\Q(U)$. We denote by
$\Sam(\M)=\Sam(\covectors)=\bigcup_{X\in \covectors}\{ S\in \{\pm 1, 0\}^U: S\leq X\}$
the \emph{set of realizable samples} for $\M$ (this is
called the \emph{polar complex} in neural codes~\cite{ItKuRo}). Since for any $X\in
\covectors$ there exists $T\in \topes$ such that $X\leq T$, we have
$\Sam(\M)=\Sam(\topes)$, see Figure~\ref{fig:treillis}.

%
%

\begin{figure}[htb]
	\centering
	\includegraphics[width=0.8\linewidth]{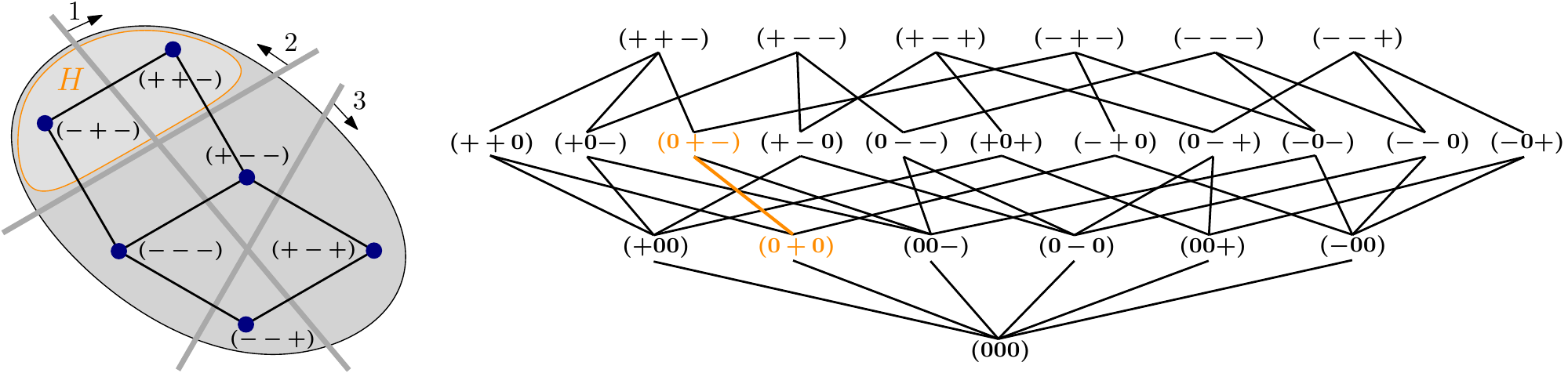}
	\caption{Left: the tope graph $G$ of the pc-restriction
	$\M$ of the COM from Figure~\ref{fig:COMrealisable} to $\{ 1,2,3\}$ and a
		convex subgraph $H$ of $G$. Right:  the realizable samples
		of $\M$ and the interval $I(H)$ (in orange).
		\label{fig:treillis}
	}
\end{figure}

Extending the notation for covectors and their faces, for a sample $S\in \Sam(\M)$
we set   $\face(S)=\{ X\in
\covectors: S\leq X\}$ and let $[S]$ be the subgraph of $G$ induced by all
topes $T\in \covectors$ from $\face(S)$. For OMs, the set $\face(S)$ is called
a \emph{supertope} in~\cite{Hochstattler}. For COMs,  $\face(S)$ is called the
\emph{fiber} of $S$ and it is known that they are COMs~\cite{BaChKn}.
Since for any  $S\in \Sam(\M)$ there exists $T\in \topes$ such that $S\leq T$,
$[S]\neq\varnothing$. Moreover, $[S]$ is the intersection of halfspaces of $G$ of
the form $G^+_e$ if $S_e=+1$ and $G^-_e$ if $S_e=-1$. Hence, $[S]$ is a
nonempty convex subgraph of $G$ for all $S\in \Sam(\M)$. 

Any convex subgraph $H$ of a partial cube $G$ is the intersection of all
halfspaces of $G$ containing $H$.
Similarly to the fact that any polytope $P$ in Euclidean space is the intersection
of the halfspaces defined by its facet-defining  hyperplanes, any convex set $H$
in a partial cube is the intersection of the halfspaces defined by the $\Theta$-classes in $\osc(H)$.
Both for $P$ and for $H$, this is a minimal representation as the intersection of halfspaces.
However, $H$ can be represented in different ways as the intersection of halfspaces. Indeed, any
representation of $H$ as an intersection of halfspaces of $G$ yields a
realizable sample $S$, where $S_e=\pm 1$ if $G_e^{\pm}$ participates in the
representation and $S_e=0$ otherwise. Notice that the $\Theta$-classes
osculating with $H$ have to be part of every representation of $H$ and the
$\Theta$-classes crossing $H$ take part in no representation of $H$.
This leads
to two canonical representations of $H$, one using only the halfspaces whose
$\Theta$-class osculates with $H$ and one using all halfspaces containing $H$:

\begin{minipage}{0.52\linewidth}
	$$(S_\bot)_e=\begin{cases}
	-1 & \text{ if } e\in\osc(H) \text{ and } H\subseteq G_e^-,\\
	+1 & \text{ if } e\in\osc(H) \text{ and } H\subseteq G_e^+,\\
	0 & \text{ otherwise.}
	\end{cases}
	$$
\end{minipage}
\hfill
\begin{minipage}{0.02\linewidth}
	and
\end{minipage}
\hfill
\begin{minipage}{0.40\linewidth}
	$$(S^\top)_e=\begin{cases}
	-1 & \text{ if } H\subseteq G_e^-,\\
	+1 & \text{ if } H\subseteq G_e^+,\\
	0 & \text{ otherwise}.
	\end{cases}
	$$
\end{minipage}

Note that $(S^\top)^0=\cross(H)$ and  $(S_\bot)^0=U\setminus \osc(H)$, i.e.,
$(S_\bot)^0$ consists of all $e$ such that $E_e$ crosses or is disjoint from
$H$. If $S$ is a sample arising from the representation of $H$ as the
intersection of halfspaces, then $S_\bot\le S\le S^\top$. Moreover, any sample
$S$ from the order interval $I(H):=[S_\bot,S^\top]$ arises from a
representation of $H$, i.e., $[S]=[S_\bot]=[S^\top]=H$. Thus, for any convex
subgraph $H$ of $G$ the set of all $S\in \Sam(\M)$ such that $[S]=H$
is an interval $I(H)=[S_\bot,S^\top]$ of  $(\Sam(\M),\leq)$.  Note
that the intervals $I(H)$ partition $\Sam(\M)$. See Figure~\ref{fig:treillis} for an illustration of the above. Moreover:



\begin{lemma}\label{lem:convex-samples1}  If $S,S'\in \Sam(\covectors)$ and
$S\le S'$, then $[S']\subseteq [S]$. 
\end{lemma}



\begin{definition} [Full samples] \label{def:full}  We say that a  realizable sample $S\in \Sam(\M)$ is \emph{full} if the pc-minor
$G'=\pi_{S^0}(G(\M))$  obtained from $G(\M)$ by contracting the $\Theta$-classes of $S^0$
has VC-dimension $d=\vcd(\M)$. Let  $\Sam_f(\M)$ denote the set of all
full samples of $\M$.
\end{definition}

Note that all topes of $\M$ are full
samples since their
zero set is empty. A convex subgraph $H$ of $G$ is \emph{full} if the sample $S_{\bot}$
is full, where recall $I(H)=[S_{\bot},S^{\top}]$. 
The image of $H$ in $G'$ (obtained from $G$ by contracting the $\Theta$-classes of $(S_\bot)^0=U\setminus \osc(H)$) is a single vertex $v_H$ and its degree  is
$|\osc(H)|$. 
If $D\subset \osc(v_H)=\osc(H)$ of size $d$ is shattered by
$G'$, since $G'$ is a pc-minor of $G$,  $D$ is also shattered by $G$. Hence, a
convex set $H$ of $G$ is full if and only if $G$ shatters a subset  $D$ of
$\osc(H)$ of size $d=\vcd(G)$. Since for any $S\in I(H)$ we have $(S^\top)^0\subseteq S^0\subseteq (S_\bot)^0$, if $H$ is a full convex subgraph of a COM, then all samples
in $I(H)$ are full. However, if $S$ is a full sample and $H=[S]$, then not necessarily all samples from $I(H)$ are full:

\begin{example}\label{example:full} Let $\M=(U,\covectors)$ be the COM  with $\vcdim(\M)=2$ defined on  $U=\{1,2,3,4,5\}$ and whose tope graph consists of one edge on each of whose ends there is a pending 4-cycle. Formally, $\M$ has
$(-,-,-,-,-), (+,-,-,-,-),(-,+,-,-,-),(+,+,-,-,-),(+,+,+,-,-),(+,+,+,+,-),$ $(+,+,+,-,+),(+,+,+,+,+)$  as topes. The two 4-cycles are the convex sets $H_1=[S_1]$ and $H_2=[S_2]$ defined by the samples  $S_1=(0,0,-,-,-)$ and  $S_2=(+,+,+,0,0)$, while  the middle edge is the convex set
$H_3=[S_3]$ where $S_3=(+,+,0,-,-)$. Consider the samples $S_\bot$ and $S^\top$ for the convex set $H_1$:  $S_{\bot}=(0,0,-,0,0)$ and $S^{\top}=(0,0,-,-,-)=S_1$. Notice that the sample $S_1=S^{\top}$ is full since contracting $S^0_1=\{ 1,2\}$ does not affect the other 4-cycle $H_2$.
However, the convex set $H_1=[S_1]$ is not full because  the sample
$S_{\bot}$ is not full:  contracting $(S_{\bot})^0=\{ 1,2,4,5\}$,
both 4-cycles will be contracted, thus the VC-dimension will decrease. Morally,
being full is a local property in a COM.
\end{example}

We show next, that this problem does not arise in OMs.

\begin{lemma} \label{lem:full-sample-OM}
	Let $\M=(U,\covectors)$ be an OM of rank $d$ and let $G=G(\M)$ be its tope
	graph. A sample $S\in \Sam(\M)$ is full if and only if the convex
	subgraph $[S]$ is full.
\end{lemma}

\begin{proof} First notice that since in OMs the rank and the VC-dimension are equal, a
	sample $S$ is full if and only if $\rk(\M\setminus S^0)=d$.
	Let $S$ be a full sample, $H=[S]$, and recall that $(S_{\bot})^0$ equals $\cross(H)$ plus
    the $\Theta$-classes not osculating with $H$. We have to show that $\M\setminus (S_{\bot})^0$
    has rank $d$. First, let $\M'=\M\setminus \cross(H)$ and let
	$G'=\pi_{\cross(H)}(G)$ be its tope graph. Since $\cross(H)\subseteq S^0$
	and $S$ is full, $\rk(\M')=d$ and hence $\vcd(G')=d$. The image of $H$ in
	$G'$ is a single vertex $v_H$. By Lemma~\ref{lem:oscilate},
	$\osc(v_H)=\osc(H)$. By Lemma~\ref{lem:localshattering},  $\osc(v_H)$
	contains a subset of size $d$ shattered by $\M'$. Since $\osc(v_H)\cap (S_{\bot})^0=\varnothing$ we conclude that $H$ is full.
	Conversely, if $H=[S]$ is a convex subgraph of $G$ that is full, then from
    the discussion preceding Example~\ref{example:full}
	we deduce that all samples from $I(H)$ (and in particular, $S$) are full.
	%
%
%
%
%
\end{proof}

\subsection{The samples $\hS$ and $\hhS$}

For a covector $X\in\covectors$ of a COM $\M=(U,\covectors)$, let $\Sam(\face(X))$ denote the samples of OM $(U,\face(X))$, i.e.,
$\Sam(\face(X))=\bigcup_{Y\in \face(X)}\{ S\in \{\pm 1,0\}^U: S\leq Y\}$. Clearly $\Sam(\face(X))\subseteq \Sam(\covectors)$.
We also denote by $\Sam(\M(X))$ the samples of the simple OM $\M(X)=(U\setminus \uX, F(X)\setminus \uX)$, i.e.,
$\Sam(\M(X))=\bigcup_{Y\in \face(X)\setminus \uX}\{ S\in \{\pm 1,0\}^{U\setminus\uX}: S\leq Y\}$. Finally, denote by $\Sam_f(\face(X))$ the set of
full samples from  $\Sam(\face(X))$ and by $\Sam_f(\M(X))$ the set of full samples from $\Sam(\M(X))$.

As we noticed already, $F(X)$ is an OM, however it is not simple. Since all covectors from $F(X)$ have the values as $X$ on the
coordinates of $\uX$, from the definition of $\M(X)$ we conclude that $F(X)=(F(X)\setminus \uX)\times X_{|\uX}$.
This establishes a one-to-one correspondence $\varphi_X$ between the covectors of $F(X)$ and the covectors of $\M(X)$ and between the topes of $\face(X)$ and the topes
of $\M(X)$. Recall that by Lemma \ref{minor-pc-minor}  the tope graph $[X]$ of $F(X)$ is isomorphic to the tope graph $G(\M(X))$ of $\M(X)$.
Therefore, $[X]$ and $G(\M(X))$ have the same sets of convex subgraphs. This one-to-one correspondence $\varphi_X$ also shows that  any $S\in \Sam(\face(X))$ has the form
$S=S'\times S''$, where $S'\in \Sam(\M(X))$ and $S''\in \{\pm 1,0\}^{\uX}$ such that $S''\le X_{|\uX}$.

The following samples $\hS$ and $\hhS$ will be important in what follows:

\begin{definition} [$\hS$ and $\hhS$] \label{def:hS} For a sample $S\in \Sam(\covectors)$ and a covector $X\in \covectors$,
set $\hS:=X\circ S$ and $\hhS:=\hS\setminus\uX=S\setminus \uX$, where it will usually be clear which covector $X\in \covectors$ we are referring to.
\end{definition}

From the definition it immediately follows that $\hS$ and $\hhS$ have the same zero sets: $\hS^0=\hhS^0$. We continue with the following properties of $\hS$ and $\hhS$:

\begin{lemma} \label{lem:convex-samples2}  Let $X\in \covectors,$   $S\in
\Sam(\M)$,  and $\Sep(X,S)=\varnothing$. 
Then  $\hS\in \Sam(\face(X)), \hhS\in \Sam(\M(X))$, the convex subgraphs $[\hS]$ of $[X]$ and $[\hhS]$ of $G(\M(X))$ are $U$-isomorphic, and
$[\hS]=[X]\cap [S]\ne \varnothing$.
\end{lemma}


\begin{proof} Since  $S\in \Sam(\M)$, there exists $Y\in \covectors$ such that $S\le Y$.
Then $\hS=X\circ S\leq X\circ Y$. Since $X\circ Y\in \face(X)$, we get $X\circ S\in \Sam(\face(X))$.
Since $(X\circ Y)\setminus \uX\in \M(X)$ and $\hhS=\hS\setminus \uX\le (X\circ Y)\setminus \uX$, we also deduce that
$\hhS\in \Sam(\M(X))$. From the definition of the convex subgraphs $[\hS]$ and $[\hhS]$ and the way how the bijection $\varphi_X$  between the topes
of $[X]$ and the topes of $G(\M(X))$ is defined, we conclude that the convex subgraphs $[\hhS]$ and $[\hS]$ are $U$-isomorphic.

Now, we prove that $[\hS]=[X]\cap [S]\ne \varnothing$.
Since $\hS=X\circ S$, we have $X\le \hS$ and by Lemma
	\ref{lem:convex-samples1} we have $[\hS]\subseteq [X]$. Now we prove that
	$[\hS]\subseteq [S]$. Indeed, otherwise there exists a tope $T$ of
	$\covectors$ such that $T\in [\hS]\setminus [S]$.
	This implies that $\hS\le T$ and there exists an element $e\in U$ such that
	$T_e\ne S_e\ne 0$. If $X_e=0$, then $\hS_e=(X\circ S)_e=S_e\ne 0$.
    Since $\hS\le T$ this implies that $S_e=\hS_e=T_e$, a contradiction with
    the
    choice of $e$.
    Otherwise, if $X_e\ne 0$, then $\hS_e=(X\circ S)_e=X_e$. This implies that $T_e=X_e$, which
	is impossible because $T_e\ne S_e\ne 0$ and we have $\Sep(X,S)=\varnothing$. This proves that
	$[\hS]\subseteq [X]\cap [S]$. Consequently, $[X]\cap [S]\ne \varnothing$. To prove the converse inclusion $[X]\cap [S]\subseteq [\hS]$ pick any tope
	$T$ of $\covectors$ belonging to $[X]\cap [S]$. Then $X\le T$ and $S\le T$ and thus $\hS=X\circ S\le T$.
	This implies that $T\in  [\hS]$ and we are done.
%
%
\end{proof}

\begin{lemma} \label{lem:convex-samples3}  Let $X\in \covectors, S\in
\Sam(\M)$, and $\Sep(X,S)=\varnothing$. Then: 

\begin{itemize}
\item[(i)] $\osc([\hhS])=\osc([\hS])=\osc([S])\cap X^0$, where $\osc([\hS])$ is considered in $[X]$ and $\osc([\hhS])$ in $G(\M(X))$; 
\item[(ii)] $\hhS^0=\hS^0=S^0\cap X^0$.
\end{itemize}
\end{lemma}

\begin{proof} To prove (i), first notice that since $[\hhS]$ and $[\hS]$ are $U$-isomorphic by Lemma \ref{lem:convex-samples2} and $[X]$ and $G(\M(X))$ are $U$-isomorphic by Lemma \ref{minor-pc-minor},
we obtain that $\osc([\hhS])=\osc([\hS])$.

Now we show that $\osc([\hS])\subseteq \osc([S])\cap X^0$. Pick any $e\in \osc([\hS])$.
Since $[\hS]\subseteq [X]$ and $[X]$ is isomorphic to $G(\M(X))$ by Lemma \ref{minor-pc-minor}, the $\Theta$-class $E_e$ necessarily crosses $[X]$, whence $e\in X^0$.
Since $e\in \osc([\hS])$, either $e\in \osc([S])$ and we are done, or $e\in \cross([S])$. Suppose by way of contradiction that
$e\in \cross([S])$. Then there exist two edges $T_1T_2$ and $T'_1T'_2$ of $E_e$ such that  $T_1\in [\hS]$ and $T_2\in [X]\setminus [\hS]$
and $T'_1,T'_2\in [S]$. But then  $T_2$ belongs to the interval either between $T_1$ and $T'_2$ or between $T_1$ and $T'_1$, contradicting the convexity of $[S]$.
This proves that $e\in \osc([S])$, establishing the inclusion
$\osc([\hS])\subseteq \osc([S])\cap X^0$.

Conversely, pick any $e\in \osc([S])\cap X^0$. Then there exist two edges $T_1T_2$ and $T'_1T'_2$ of $E_e$, such that  $T_1,T_2\in [X]$ and $T'_1\in [S]$, $T'_2\notin [S]$.
Since $X\in \covectors$, $[X]$ is gated. Denote by $T''_1$ and $T''_2$ the gates of respectively $T'_1$ and $T'_2$ in $[X]$: $T''_1=X\circ T'_1$ and $T''_2=X\circ T'_2$.
Since $T'_1$ and $T'_2$ are adjacent, the topes $T''_1$ and $T''_2$ are either adjacent or coincide. Furthermore, since $T''_1$ belongs to the interval $I(T'_1,T)$
between $T'_1$ and any $T\in [S]\cap [X]\ne \varnothing$, the convexity of $[\hS]=[S]\cap [X]$ implies that $T'_1\in [\hS]$. Now, if $T''_2=T''_1$, since $T'_1\in [S]$ and $T'_2\notin [S]$, the convexity of $[S]$
implies that $T'_1$ is in the interval $I(T'_2,T''_2)=I(T'_2,T''_1)$. Since $T''_2$ is in the intervals $I(T'_2,T_1)$ and $I(T'_2,T_2)$, we conclude that $T'_1$ also belongs to the intervals $I(T'_2,T_1)$ and $I(T'_2,T_2)$.
But this is impossible because the edges $T'_1T'_2$ and $T_1T_2$ belong to the same $\Theta$-class $E_e$. This proves that $T''_1$ and $T''_2$ are different and adjacent. Moreover, $T''_1\in I(T'_1,T''_2)$ and $T''_2\in I(T'_2,T''_1)$,
proving that the edge $T''_1T''_2$ also belongs to the $\Theta$-class $E_e$. Then we also have $T'_1\in I(T'_2,T''_1)$ and $T'_2\in I(T'_1,T''_2)$. Since $T'_1\in [S]$ and $T'_2\notin [S]$, the convexity of $[S]$ implies that
$T''_2\notin [S]$. Consequently, $T_1''\in [S]\cap [X]=[\hS]$ and $T''_2\in [X]\setminus [S]=[X]\setminus [\hS]$, establishing that $e\in \osc([\hS])$. This proves the inclusion $\osc([S])\cap X^0\subseteq \osc([\hS])$ and
concludes the proof of (i).

To prove (ii), first notice that $\hhS^0=\hS^0$ and that $\hS^0\subseteq S^0\cap X^0$. To prove the converse inclusion, pick any $e\in S^0\cap X^0$. Then there exist two edges $T_1T_2$ and $T'_1T'_2$ of $E_e$, such that  $T_1,T_2\in [X]$ and $T'_1,T'_2\in [S]$.
As in previous proof, let  $T''_1$ and $T''_2$ be the gates of  $T'_1$ and $T'_2$ in $[X]$. Then as above we deduce that $T''_1T''_2$ is an edge of $E_e$ belonging to $[X]$. If $T$ is a tope of $[\hS]=[S]\cap [X]$ (such a tope exists by Lemma~\ref{lem:convex-samples2}),
then $T''_1\in I(T'_1,T)$ and $T''_2\in I(T'_2,T)$. Since $[S]$ is convex and $T'_1,T'_2\in [S]$, we conclude that $T''_1,T''_2\in [S]$. Consequently, $T''_1T''_2$ is an edge of $[\hS]$, hence $e\in \hS^0$, establishing (ii).
\end{proof}

\subsection{Lower and upper covectors} Let $\M=(U,\covectors)$ be a COM.  We define lower and upper covectors for samples of $\M$. For a sample $S\in \Sam(\M)$ of $\M$ consider the
tope $T'=S\setminus S^0$ of $\M':=\M\setminus S^0$.  Any minimal non-zero covector
$X'$ of $\M'$ such that $T'\geq X'$ is called a \emph{lower covector} for $S$.
Since $\M'$ is a COM and $T'$ is a tope of $\M'$, lower covectors $X'$ for $S$ exist.
Any covector of $\M$ such that $X\setminus S^0=X'$
is called an \emph{upper covector} for $S$. Again, upper covectors for $S$ exist because
$X'$ is the restriction of some covector $X$ of $\M$. 
Note that
if $\M$ is an OM, then the lower and upper covectors are always cocircuits, which we will sometimes call \emph{lower and upper cocircuits} for $S$. For lower covectors this follows by minimality, but for upper covectors this follows from $S$ being full and is part of 
Lemma~\ref{cfocircuit-full-bis}.

Recall that we denote by $\M'(X')=\M'\setminus\underline{X'}$ the simple OM defined by the face
$\face(X')$ of $\M'$ and by $\M(X)=\M\setminus\uX$
the simple OM defined by the face $\face(X)$ of $\M$.

\begin{lemma} \label{upper0} If $S\in \Sam(\M)$, $X'$ is a lower covector for $S$, and $X\in \covectors$ is an upper covector for $S$ such that $X\setminus S^0=X'$, then
$\Sep(S,X)=\varnothing$ and $\vcdim(X)\ge \vcdim(X')$. Furthermore, if $\vcdim(X)=\vcdim(X')$, then $\hhS$ 
is a full sample of $\M(X)$.
\end{lemma}

\begin{proof} Let $X'$ be a lower covector for $S$.  Since $X'=X\setminus S^0$ and $X'\le T'=S\setminus S^0$, for any tope
$T$ of $\M$ such that $T'=T\setminus S^0$ (such tope $T$ exists since $\M$ is simple), we  have $S\leq T$ and $X\le T$,
yielding $T\in [S]\cap [X]$. Thus $\Sep(S,X)=\varnothing$. 

Now we prove that $\vcdim(X)\ge \vcdim(X')$. Let $\vcdim(X')=d$. By Lemma
\ref{lem:localshattering}, there exists a set $D\subseteq \osc([T'])\cap X'^0$
of size $d$
shattered by $\M'(X')$. Since the tope graph of $\M'(X')$ is a pc-minor of $G(\M)$, $D$ is shattered by $\M$. Since $D\subseteq X'^0\subseteq X^0$ and $[X]$ is a gated subgraph of the tope graph
of $\M$, by Lemma~\ref{lem:gatedshatter}, $D$ is shattered by  $\M(X)$. This shows that $\vcdim(X)\ge d=\vcdim(X')$.

Now suppose that $\vcdim(X)=d$ and we assert that $\hhS$ is a full sample of $\M(X)$. By Lemma~\ref{lem:full-sample-OM} applied to OM $\M(X)$, the sample $\hhS$ is full if and only if the convex set $[\hhS]$ is full.
Since $\Sep(S,X)=\varnothing$, by Lemma~\ref{lem:convex-samples3}(i),
$\osc([\hhS])=\osc([S])\cap X^0$. Since  $D\subseteq \osc([T'])\cap X'^0$ and
$\osc([T'])=\osc([S])$ (by Lemma~\ref{lem:oscilate}), $X'^0\subseteq X^0$, we
deduce that
$D\subseteq \osc([\hhS])$.
Consequently, $\M(X)$ shatters a set $D\subseteq \osc([\hhS])$ of size $d$,
establishing that the convex set $[\hhS]$ is full in $\M(X)$.
\end{proof}

\begin{lemma} \label{upper1} Let $S\in \Sam(\M)$, $X'$ be a lower
covector for $S$, and $X\in \covectors$ be an upper covector for $S$
such that $X'=X\setminus S^0$. Then $\M'(X')=\M(X)\setminus\hhS^0=\M(X)\setminus\hS^0$. Consequently, 
$\vcd(X)\ge \vcd(X')$.
\end{lemma}

\begin{proof} First we prove the following claim:

\begin{claim} \label{equality} $\underline{X'}\cup S^0=\uX\cup \hS^0$.
\end{claim}

\begin{proof} To prove the inclusion $\underline{X'}\cup S^0\subseteq \uX\cup \hS^0$ notice that $\underline{X'}\subseteq \uX$ by the definition of $X$.
If $e\in S^0\setminus \uX$, then $e\in X^0$. By Lemma~\ref{lem:convex-samples3}(ii), $e\in S^0\cap X^0=\hS^0$, establishing that
$\underline{X'}\cup S^0\subseteq \uX\cup \hS^0$. To prove the converse inclusion $\uX\cup \hS^0\subseteq \underline{X'}\cup S^0$ note that $\hS^0\subseteq S^0$. If
$e\in \uX\setminus S^0$, then $e\in \underline{X'}$ because $X'=X\setminus S^0$, and we are done.
\end{proof}

Denote by $G(\M),G(\M'),$ and $G(\M(X))$ the tope graphs of $\M, \M'=\M\setminus S^0$, and $\M(X)$, respectively. Denote also by $G'$ the tope graph of $\M'(X')$ and by $G''$ the tope graph of $\M(X)\setminus\hS^0$. To prove that $\M'(X')=\M(X)\setminus\hS^0$ it suffices to establish that the tope graphs $G'$ and $G''$ coincide.
By Lemma~\ref{minor-pc-minor}, $[X]$ is isomorphic to $G(\M(X))=\pi_{\underline {X}}(G(\M))$. Furthermore, by the same lemma, $G''=G(\M(X)\setminus\hS^0)=\pi_{\hS^0}(G(\M(X))$.
Consequently, $G''=\pi_{\uX\cup \hS^0} G(\M)$. Analogously, by Lemma~\ref {minor-pc-minor}, $G'=G(\M'(X'))=\pi_{\underline{X'}}(G(\M'))$ and is isomorphic to $[X']$.
Since $G(\M')=\pi_{S^0}(G(\M))$, we conclude that $G'=\pi_{S^0\cup \underline{X'}}(G(\M))$. By Claim~\ref{equality}, $\underline{X'}\cup S^0=\uX\cup \hS^0$. Since the pc-contractions commute, we obtain that  $$G'=\pi_{S^0\cup \underline{X'}}(G(\M))=\pi_{\uX\cup \hS^0}(G(\M))=G'',$$ whence $\M'(X')=\M(X)\setminus\hS^0$. Since $\hhS^0=\hS^0$, we obtain the equality $\M'(X')=\M(X)\setminus\hS^0=\M(X)\setminus\hhS^0$. Since $G'=G''$ is a pc-minor of $G(\M(X))$, also  $\vcd(X)\ge \vcd(X')$ holds.
%
%
\end{proof}

In the following two results we suppose that $\M=(U,\covectors)$ is an OM of VC-dimension $d$.

	\begin{lemma} \label{cfocircuit-full}
		For any tope $T$ of $\M$ and $e\in \osc([T])$, there exists a
		cocircuit $X$ of $\M$ such that $e\in \uX,$ $X\leq T$, and
        $\M(X)$ has VC-dimension $d-1$. 
    \end{lemma}
	
	\begin{proof}
		Since $T$ is a tope and $e\in \osc([T])$,  $T$ is
		incident to an edge	of $E_{e}$, i.e., there is a tope $T'$ of $\M$
		such that $\Sep(T,T')=\{e\}$. Let $X$ be a cocircuit of $\M$ such
		that its face $\face(X)$ contains $T$ but not $T'$. This cocircuit
		$X$ exists, otherwise all cocircuits $Y$ of $\M$ would have
		$Y_{e}=0$,	contradicting the assumption that $\M$ is simple.
        Now, since $\M$ has VC-dimension $d$, $\M(X)$ has VC-dim $d-1$ by
		Lemma~\ref{lem:rankfunction}. 
		%
        Furthermore, as $T\in [X]$ and $T'\notin [X]$, we immediately
        get that $X\leq T$ and $e\in \uX$.  
   \end{proof}
%

\begin{lemma} \label{cfocircuit-full-bis}  For any full sample $S$ of $\M$ and $e\in \osc([S])$, there exists a lower cocircuit $X'$ for $S$ 
such that $e\in \underline{X'}$. For any such $X'$, there exists an upper cocircuit $X$ for $S$. Any such cocircuit $X$ satisfies that $\vcd(X)=d-1$, $e\in \uX$, $\Sep(S,X)=\varnothing$, and $\hhS$
        is a full sample of $\M(X)$.
\end{lemma}

\begin{proof}
	Since $S$ is a full sample, $\M'=\M\setminus S^0$ has rank $d$.
        Moreover, $S\setminus S^0$ is a tope $T'$ of $\M'$.
		By Lemma~\ref{lem:oscilate}, $e\in \osc([S])=\osc([T'])$ and by Lemma~\ref{cfocircuit-full}
        there exists a cocircuit $X'$ of $\M'$ such that $e\in \underline{X'}$, $X'\leq T'$, and
        $\M(X')$ has VC-dim $d-1$.  Thus $X'$ is a lower cocircuit for $S$ and hence there exists an upper
        covector  $X$ of $\M$ such that $X'=X\setminus S^0$. By Lemma~\ref{upper1}, $\vcd(X)\ge \vcd(X')=d-1$. If $X$ was not a
		cocircuit, then $\face(X)$ is a proper face of $\face(Y)$ for some cocircuit $Y$ of $\M$. Since in an OM the VC-dimension of any proper face is
		strictly smaller than the VC-dimension of the face itself  and since
		$\M$ has VC-dimension $d$, we obtain a contradiction. Thus $X$ is a cocircuit
		of $\M$ (and an upper cocircuit for $S$) and $\vcd(X)=\vcd(X')=d-1$. In particular, $e\in \uX$.
        By 
        Lemma~\ref{upper0}, $\Sep(S,X)=\varnothing$ and $\hhS$ is a full sample
        of $\M(X)$.
 %
  %
\end{proof}

\section{The main result}

The goal of this section is to prove the following theorem:

\begin{theorem} \label{main} The set $\topes$ of topes of a complex of oriented
matroids $\M=(U,\covectors)$ of VC-dimension $d$ admits a
proper labeled sample compression scheme of size $d$.
\end{theorem}


\subsection{The main idea}
Our labelled sample compression scheme takes any realizable sample $S$ of a COM $\M$ and
removes the zero set of $S$. Consequently, $S$ becomes the tope $S\setminus S^0$ of
the COM $\M\setminus S^0=:\M'$. Then we consider a face $\face(X')$ of $\M'$ defined
by a minimal covector  $X'$ of $\M'$ such that $S\setminus S^0\geq X'$ (i.e., by a lower covector for $S$).
This face defines the simple OM $\M'(X')=\M'\setminus\underline{X'}$. The compressor  $\alpha(S)$ is then
defined by applying to $\M'(X')$ and its tope $S\setminus (S^0\cup \underline{X'})$
the \emph{distinguishing lemma}, which
allows to distinguish the  full samples of an OM $\M$ of rank $d$ by considering
their restriction to subsets of size $d$. It constructs a function $f_\M$
that assigns such a subset to each full sample and is used by both compressor
and reconstructor. The \emph{localization lemma} is used by the reconstructor and
designates the set of all potential covectors whose faces
may contain topes $T$ compatible with the initial sample $S$. These two lemmas
are proved in next two subsections.  Compressor and
reconstructor are given in the last subsection and are illustrated by
Example~\ref{ex:compression}. The compressor generalizes the compressor for ample
classes of Moran and Warmuth~\cite{MoWa}. However, the reconstructor is more involved
than that for ample classes.

\subsection{The distinguishing lemma} In this subsection, $\M=(U,\covectors)$ is an OM of VC-dimension/rank $d$. 
We continue with the definition of the function $f_\M$ defined on the set $\Sam_f(\M)$ of
full samples of $\M$. Fix a linear order on the ground set $U=\{1, \ldots, m\}$ of $\M$.
For any subset $U'=U\setminus A$ of $U$ we will consider the restriction of this linear order to $U'$. Suppose recursively
that we have already defined the functions $f_{\M'}$ on the set $\Sam_f(\M')$ of full samples
of all proper (i.e., $A\neq\varnothing$) deletions $\M'=(U\setminus A,\covectors\setminus A)$ of $\M$.
Let $S\in \Sam_f(\M)$ be a full sample of $\M$. If $S$ is not a tope of $\M$, then we set
$f_{\M}(S)=f_{\M\setminus S^0}(S\setminus S^0)$. Otherwise, if $S$ is a tope of $\M$, then  
we set $f_{\M}(S)=\{e_S, f_{\M(X')}(S\setminus e_S)\},$
where:
\begin{itemize}
\item  $e_S$ is the smallest element of $\osc([S])$;
\item $X'$ is the lexicographically minimal lower cocircuit for $S$ in $\M$ such that $e_S\in \underline{X'}$ and $X'\leq S$.
\end{itemize}
Equivalently, $f_\M(S)$ can be defined recursively by setting $f_{\M}(S)=\{e_S, f_{\M'(X')}(S\setminus(S^0\cup \underline{X'}))\},$
where $\M'=\M\setminus S^0$ and:
\begin{itemize}
\item  $e_S$ is the smallest element of $\osc([S\setminus S^0])=\osc([S])$;
\item $X'$ is the lexicographically minimal lower cocircuit for $S$ in $\M'$ such that $e_S\in \underline{X'}$ and $X'\leq S\setminus S^0$.
\end{itemize}

\begin{remark}
Here we order sign vectors lexicographically by setting $0<+<-$. This choice is needed in order to avoid freedom in the definition,
but is arbitrary. Indeed, we will prove that taking any lower cocircuit $X'$ for $S$ in $\M'$ such that $e_S\in \underline{X'}$ and $X'\leq S$ will work.
\end{remark}

The equality $\osc([S])=\osc([S\setminus S^0])$ holds by Lemma~\ref{lem:oscilate}. The cocircuit $X'$ exists by Lemma~\ref{cfocircuit-full}.
Since $S\setminus (S^0\cup \underline{X'})$ is a tope (and thus a full sample) of $\M'(X')$ and since $\M'(X')$ has VC-dimension $d-1$ by Lemma~\ref{cfocircuit-full},
by induction hypothesis $f_{\M'(X')}(S\setminus (S^0\cup \underline{X'}))$ is well-defined. Furthermore, $f_\M(S)$ has size $d$, thus $f_\M$ is a map from
$\Sam_f(\M)$ to ${U \choose d}$.

Now, we define an equivalence relation $\sim$ on the set  $\Sam_f(\M)$ of all full samples of $\M$:

\begin{definition} [Equivalence classes of full samples] Two full samples $S,S'\in \Sam_f(\M)$ are \emph{equivalent}
(notation $S\sim S'$) if $f_\M(S)=f_\M(S')$ and $S_{|f_{\M}(S)}=S'_{|f_{\M}(S')}$ hold. Clearly, $\sim$ is an equivalence relation on $\Sam_f(\M)$.
Denote by $\Omega_1,\ldots,\Omega_k$ the equivalence classes of $\Sam_f(\M)$.
\end{definition}

The partition of $\Sam_f(\M)$ into equivalence classes can be also
viewed in the following way. For any set $D\subseteq U$
of size $d$ and any $C\in \{\pm 1,0\}^U$ with $ \underline{C}=D$,  we denote by
$\Omega(C,D)$ the set of all $S\in \Sam_f(\M)$ such that
$f_\M(S)=D$ and $S_{|f_{\M}(S)}=C$. Then $\Omega(C,D)$ is either empty or is an equivalence class of $(\Sam_f(\M),\sim)$.

We continue with the distinguishing lemma, which shows that $f_{\M}$  distinguishes samples from different equivalence classes of $\sim$ and defines for all samples from the same equivalence class
a nonempty convex set, which later in Definition~\ref{def:real} will be called the realizer and will be used by the reconstructor.

\begin{lemma}\label{lem:distinguish}
 Let $\M=(U,\covectors)$ be an $\OM$ of VC-dimension  $d$. The function
 $f_{\M}:\Sam_f(\M) \to {U \choose d}$ has the following properties for
 all $S\in \Sam_f(\M)$:
\begin{itemize}
\item[(i)] $f_{\M}(S)\subseteq \osc([S])$,
\item[(ii)] $f_{\M}(S)$ is shattered by $\M$,
\item[(iii)] for any equivalence class $\Omega_i, i=1,\ldots, k$ of $(\Sam_f(\M),\sim)$, $\bigcap_{S\in \Omega_i} [S]\ne \varnothing$. 
\end{itemize}
\end{lemma}

\begin{proof}
Let $G:=G(\M)$ be the tope graph of $\M$. We proceed by induction on $d$.
	If $d=1$, then $U=\{ e\}$ and $G$ is an edge between the topes $T_1=(-1)$
	and $T_2=(+1)$, which are the only full samples of $\M$. 
    Then $f_\M(T_1)=f_\M(T_2)=\{
	e\}$ and  we obtain a function satisfying the conditions (i)-(iii). Thus,
    let $d\geq 2$.

	\medskip \noindent
	\textbf{Condition (i):} By definition of $f_\M(S)$, the element $e_S$ is chosen from $\osc([S\setminus S^0])=\osc([S])$. Let $T'=S\setminus S^0$. By induction hypothesis, the remaining elements
    of $f_\M(S)$ will be chosen from  $\osc([T'\setminus \underline{X'}])$. Note that $T''=T'\setminus \underline{X'}=S\setminus (S^0\cup \underline{X'})$ is a tope of $\M'(X')$ and $\osc([T''])$ is defined by the edges of the
    tope graph of $\M'(X')$  incident to $T''$. Since this is a subset of edges incident to  $T'$ in the tope graph of  $\M'$, we conclude that $\osc([T''])\subseteq \osc([T'])=\osc([S])$. This proves that $f_\M(S)\subseteq \osc([S])$.

	\medskip\noindent
	\textbf{Condition (ii):} Suppose that $f_{\M}(S)$ is not shattered by $\M$.
	Define $D'=f_{\M'(X')}(T'\setminus\underline{X'})$, where $\M'=\M\setminus
	S^0$, $T'=S\setminus S^0$, and
    $X'$ is any cocircuit of $\M'$ such that $e_S\in \underline{X'}$ and $X'\leq T'$, which exists by Lemma
   ~\ref{cfocircuit-full-bis}.
    By the induction hypothesis, $D'$ is shattered by  $\M'(X')$. By Lemma
   ~\ref{cfocircuit-full-bis}, there exists a cocircuit $X$ of $\M$ such that
    $X\setminus S^0=X'$ and $e_S\in \underline{X'}$.
    Since $D'$ is shattered by $\M'(X')$, we get $D'\subseteq X'^0\subseteq X^0$. 
	Since $f_\M(S)=D'\cup \{ e_S\}$ is not shattered by $\M$,  by Lemma~\ref{lem:shatterandcircuits} there is a circuit $Y$ of $\M$
	such that $\uY\subseteq \{e_S\}\cup D'$ and $e_S\in
	\uY$. On the other hand,  $D'\subseteq X^0$ and
	$e_S\in \uX$, thus $|\uY\cap \uX|=1$. Since $X$ is a cocircuit and $Y$ is a circuit, this contradicts
	orthogonality of circuits and cocircuits in	$\OM$s, see
	Theorem~\ref{lem:orthogonality}.

	\medskip\noindent
	\textbf{Condition (iii):} The case $d=1$ was considered above, so let  $d\ge 2$. 
    Suppose that $\Omega_i=\Omega(C,D)$ for some $C\in \{\pm 1,0\}^U$ and $D=\underline{C}$. Let $Q,R$ be any two full samples of $\Omega(C,D)$ and denote $\M'=\M\setminus Q^0$ and $\M''=\M\setminus R^0$.
    Thus $f_\M(Q)=f_\M(R)=D$ and $Q_{|f_\M(Q)}=R_{|f_\M(R)}=C$. By definition, $f_\M(Q)=\{e_Q, f_{\M'(X'_Q)}(Q\setminus (Q^0\cup \underline{X'_Q}))\}$, where $e_Q$ is the smallest element
    of $\osc([Q\setminus Q^0])=\osc([Q])$ and $X'_Q$ is a lower cocircuit for $Q$ such that $e_Q\in \underline{X'_Q}$ and $X'_Q\leq Q\setminus Q^0$.
    Analogously, $f_\M(R)=\{e_R, f_{\M''(X'_R)}(R\setminus (R^0\cup \underline{X'_R}))\}$, where $e_R$ is the smallest element
    of $\osc([R\setminus R^0])=\osc([R])$ and $X'_R$ is a lower cocircuit for $R$ such that $e_R\in \underline{X'_R}$ and $X'_R\leq R\setminus R^0$.
    Since $f_\M(Q)=f_\M(R)$,  by the minimality in the choice of the elements
	$e_Q$ and $e_{R}$,  both are the smallest elements of the respective sets
	$f_\M(Q)$ and $f_\M(R)$. Consequently,  $e_Q=e_{R}=:e$ and $D=\{ e\}\cup
	D'$, where $f_{\M'(X'_Q)}(Q\setminus (Q^0\cup
	\underline{X'_Q}))=f_{\M''(X'_R)}(R\setminus (R^0\cup
	\underline{X'_R}))=:D'$.

    By Lemma~\ref{cfocircuit-full-bis}, there exists an upper cocircuit $X_Q$
    of $\M$ such that $X_Q\setminus Q^0=X'_Q$,  $e\in \underline{X_Q}$, and
    $\vcd(X_Q)=d-1$.
    Analogously, there exists an upper cocircuit $X_R$ of $\M$ such that
    $X_R\setminus R^0=X'_R$,  $e\in \underline{X_R}$, and $\vcd(X_R)=d-1$.
    Furthermore, by the same Lemma~\ref{cfocircuit-full-bis} and by
    Lemma~\ref{lem:convex-samples2}, we have
    $[Q]\cap [X_Q]\ne \varnothing$ and $[R]\cap [X_R]\ne \varnothing$. Since both faces $\face(
	X_Q)\cong \M(X_Q)$ and $\face(X_R)\cong\M(X_R)$ of $\M$ shatter the same set $D'\subseteq
	U$,	Lemma~\ref{lem:galleries} implies that $X_Q=X_R$ or $X_Q=-X_R$.
	Indeed, let $X_Q\ne \pm X_R$. Since $X_Q,X_R$ maximally shatter $D'$, by Lemma
	~\ref{lem:galleries}(ii) $X_Q=X_Q\circ X_R$ and $X_R=X_R\circ X_Q$. By Lemma
	~\ref{lem:galleries}(iii) there exists  a geodesic gallery between $\face(
	X_Q)$ and $\face(X_R)$. Since $X_Q$ and $X_R$ are cocircuits of $\M$, $\face(X_Q)$
	and $\face(X_R)$ are facets of $\M$. Therefore  $\face(X_Q)$ and $\face(X_R)$
	must be consecutive in the gallery and the face containing them as facets
	must coincide with $\M$. Thus, $X_Q=\pm X_R$.
	
	But if $X_Q=-X_R$ holds,  since
	$e\in\underline{X_Q}\cap\underline{X_R}$, we have $e\in\Sep(X_Q,X_R)$. Since
	$[Q]\cap [X_Q]\ne \varnothing$ and $[R]\cap [X_R]\ne \varnothing$, for any two topes
    $T'\in [Q]\cap [X_Q]$ and $T''\in [R]\cap [X_R]$ we will have $T'_e=-T''_e$.
    Since $e\in \osc([Q])\cap \osc([R])$, we get $Q_e=-R_e$,  which contradicts the
	assumption $Q_{|f_\M(Q)}=R_{|f_\M(R)}$. Hence, $X_Q=X_R$. Since the equality $X_Q=X_R$ holds for any $Q,R\in \Omega(C,D)$, there exists a cocircuit $X$ of $\M$ such that
    for any $S\in \Omega(C,D)$, we have $X\setminus S^0=X'_S$,  $e\in \underline{X'_S}$, $\vcd(X)=d-1$, and $[S]\cap [X]\ne \varnothing$.

    By the induction hypothesis, the function $f_{\M(X)}$ defined on the set $\Sam_f(\M(X))$
    of full samples of $\M(X)$ satisfies the properties (i)-(iii) of the lemma. Let $C'$ denote the restriction of $C$ to $D'$.
    Denote by $\Omega'(C',D')$ the set of all $Q'\in \Sam_f(\M(X))$ such that $f_{\M(X)}(Q')=D'$ and $Q'_{|D'}=C'$.
    For any $Q\in \Omega(C,D)$, we have $[Q]\cap [X]\ne\varnothing$, thus $\Sep(X,Q)=\varnothing$. By Lemma~\ref{lem:convex-samples2}, $[\hQ]=[Q]\cap [X]\ne \varnothing$. By the same lemma, $\hhQ\in \Sam(\M(X))$ and $[\hhQ]$ is  $U$-isomorphic to $[\hQ]$.
    By Lemma~\ref{cfocircuit-full-bis}, $\hhQ$ is a full sample of $\M(X)$,
    i.e.,  $\hhQ\in \Sam_f(\M(X))$. We assert that $\hhQ\in \Omega'(C',D')$.
    Recall that $X$ is an upper cocircuit for $Q$ and $X'_Q$ is a lower
    cocircuit for $Q$ such that $X\setminus Q^0=X'_Q$. By Lemma~\ref{upper1},
    $\M'(X_Q')=\M(X)\setminus\hQ^0=\M(X)\setminus \hhQ^0$. This implies that
    $f_{\M(X)}(\hhQ)=f_{\M'(X'_Q)}(Q\setminus (Q^0\cup \underline{X'_Q}))=D'$
    and consequently that $\hhQ_{|{D'}}=Q_{|D'}=C'$.
    This establishes the inclusion $\{ \hhQ: Q\in \Omega(C,D)\}\subseteq
    \Omega'(C',D')$. Since $\Omega(C,D)=\Omega_i\ne \varnothing$,  the set
    $\Omega'(C',D')$ is nonempty and thus is an equivalence class of  $(\Sam_f(\M(X)),\sim)$.
    By the induction hypothesis, in $G(\M(X))$ we have $\bigcap_{Q'\in \Omega'(C',D')} [Q'] \ne \varnothing$.
    Denote this intersection by ${\mathcal R}'(C',R')$.

    Let ${\mathcal R}(C,R)$ denotes the (nonempty) set of topes $T$ of $G(\M)$ of the form $T=T'\times X_{|\uX}$ for some tope $T'$ of $\M(X)$ belonging to the set ${\mathcal R}'(C',R')$.
    By the  one-to-one correspondence $\varphi_X$ between the topes of $[X]$ and the topes of $G(\M(X))$ we conclude that ${\mathcal R}(C,R)\subseteq [X]$. Pick any sample $Q\in \Omega(C,D)$.
    Since $\hhQ\in \Omega'(C',D')$, we get $T'\in {\mathcal R}'(C',R')\subseteq [\hhQ]$ (recall that $[\hhQ]$ is considered in $G(\M(X))$).
    By the $U$-isomorphism between the convex subgraphs $[\hQ]$ and $[\hhQ]$ (Lemma \ref{lem:convex-samples2}), we deduce that the tope $T=T'\times X_{|\uX}$  belongs  in $G(\M)$ to $[\hQ]$.
    Consequently, the inclusion ${\mathcal R}(C,R)\subseteq [\hQ]\cap [X]$ holds for any $Q\in \Omega(C,D)$.  Since for any $Q\in \Omega(C,D)$,  $[\hQ]=[Q]\cap [X]$ by Lemma~\ref{lem:convex-samples2},
    we conclude that $(\bigcap_{Q\in \Omega(C,D)} [Q])\cap [X]=\bigcap_{Q\in \Omega(C,D)} ([Q]\cap [X])\supseteq {\mathcal R}(C,R)\ne \varnothing$. Consequently, $\bigcap_{Q\in \Omega(C,D)} [Q]\ne \varnothing$. 
     This concludes the proof of property (iii) and of the lemma.
   \end{proof}

 \begin{definition} [Realizers] \label{def:real}  For an equivalence class $\Omega_i=\Omega(C,D)$ of $(\Sam_f(\covectors),\sim)$, we call the nonempty intersection ${\mathcal R} (C,D)=\bigcap_{S\in \Omega_i} [S]$ the \emph{realizer}
 of $\Omega(C,D)$.
 \end{definition}

\subsection{The localization lemma}
The \emph{localization lemma} designates for any realizable sample $S$
of a COM $\M$ the set of all potential covectors 
whose faces 
may contain topes 
of $\M$ which can be used by the reconstructor.

Let $\M=(U,\covectors)$ be a COM of VC-dimension $d$ and let $S\in
\Sam(\M)$ be a realizable sample. Consider the tope $T'=S\setminus
S^0$ of the COM $\M':=\M\setminus S^0$ and let $X'$ be a minimal covector of
$\M'$ such that $T'\geq X'$. 
By Lemma~\ref{lem:rankfunction}, the $\OM$ $\M'(X')=
\M'\setminus \underline{X'}$ has VC-dimension $\le d$. Let \[{\mathcal
H}_{S,X'}:=\{X\in\covectors: X\setminus S^0=X' \text{ and $\vcdim(\M(X))=\vcdim(\M'(X'))$}\}.\]
For a set $D\subseteq U$, let \[{\mathcal H}_D:=\{X\in\covectors: \M(X) \text{
maximally shatters } D\}.\]

\begin{lemma}\label{lem:HS=HD}
	Let $S\in\Sam(\M)$, $X'$ be a minimal covector of
	$\M'=\M\setminus S^0$ such that $S\setminus S^0=T'\geq X'$, and let
	$D$ be a subset of $\underline{S}=U\setminus S^0$ such that $|D|=\vcdim(\M'(X'))$ and $D$ is shattered by $\M'(X')$. Then $\varnothing\ne
	{\mathcal H}_{S,X'}={\mathcal H}_D$.
\end{lemma}

\begin{proof} First, we prove that ${\mathcal H}_{S,X'}\subseteq
	{\mathcal H}_D$. Pick any $X\in {\mathcal H}_{S,X'}$. Since
$\M'(X')$ shatters $D$ and  $G(\M'(X'))$ is a pc-minor of
$G(\M(X))$ because $X\setminus S^0=X'$, $\M(X)$ also shatters $D$.
Since $\vcdim(\M(X))=\vcdim(\M'(X'))$, $\M(X)$ maximally shatters $D$, yielding
$X\in {\mathcal H}_D$.

Now we prove that the set ${\mathcal H}_{S,X'}$ is nonempty. By Lemma~\ref{lem:deletecocircuits} there
exists at least one covector $X\in \covectors$ such that $X\setminus S^0=X'$. For the same reason as above,
$\M(X)$ shatters $D$. Suppose that $\M(X)$ shatters
a superset of $D$. By Lemma~\ref{lem:galleries}(iii), there exists a covector $Y>X$ of $\M$ such that $\M(Y)$ maximally shatters $D$.
Hence, $Y\setminus S^0\geq X\setminus S^0=X'$, but $\M'(Y\setminus S^0)$ and $\M'(X')$ have the same VC-dimension since they both maximally shatter the set $D$.
By Lemma~\ref{lem:rankfunction}, $Y\setminus S^0=X'$ and hence $Y\in {\mathcal H}_{S,X'}$. This proves that ${\mathcal H}_{S,X'}\ne\varnothing$.
%
	
	It remains to prove that ${\mathcal H}_D\subseteq {\mathcal H}_{S,X'}$. Assume by way of contradiction
    that there exists  $Y\in {\mathcal H}_D\setminus {\mathcal H}_{S,X'}$ and set $Y'=Y\setminus
	S^0$. Since $Y\notin {\mathcal H}_{S,X'}$ and $\M(Y)$ maximally shatters $D$, we have $X'\neq Y'$.
    Since $\M(Y)$ maximally shatters $D$ and $D\subseteq \underline{S}$, also $\M'(Y')$ maximally shatters $D$. In particular,
    $D\subseteq X'^0\cap Y'^0=(X'\circ Y')^0$. By
	Lemma~\ref{prop:projection} the gates of $[Y']$ in $[X']$ are the topes of
	$\face(X'\circ Y')\subseteq \face(X')$. Thus, $[X'\circ Y']$ is a gated
	subgraph of $[X']$, and $[X'\circ Y']$ is crossed by $D$ (since $D\subseteq (X'\circ Y')^0=\cross([X'\circ Y'])$), and $D$ is
	shattered by $[X']$. By Lemma~\ref{lem:gatedshatter}, the VC-dimension of
	$\M'(X'\circ Y')$ is at least $|D|$, which is the VC-dimension of
	$\M'(X')$. Then Lemmas~\ref{lem:galleries}(i) and~\ref{lem:rankfunction} yield $X'\circ Y'=X'$. If $\Sep(X',Y')=\varnothing$, then $\face(X')=\face(X'\circ Y')\subseteq
	\face(Y')$.  Since $\face(X')$ is a maximal face of $\M'$, we get $X'=Y'$.
	Otherwise, if $\Sep(X',Y')\neq \varnothing$, then by Lemmas~\ref{lem:galleries}(iii)  and~\ref{prop:projection}   there
	exists a geodesic gallery $(\face(X')=\face(X_0), \face(X_1), \ldots,
	\face(X_k)=\face(Y'))$  with $k>0$  from $\face(X')$ to $\face(Y')$ in $\M'$. By the
	definition of a gallery, the union of $\face(X')$ and $\face(X_1)$ is included in a
	face $\face(Z)\supsetneq \face(X')$ of $\M'$. Thus, $\face(X')$ is not a
	maximal face of $\M'$, contradicting the assumption that $X'$ is a
	minimal covector of $\M'$.
\end{proof}

\subsection{The labeled compression scheme}
Now, we describe the compression and the reconstruction and prove their
correctness.  The compression map 
generalizes the compression map for ample classes of~\cite{MoWa}. However, the
reconstruction map 
is much more involved than the reconstruction map for ample classes, since it uses
both the distinguishing and the localization lemma.

\medskip\noindent
\textbf{Compression.}
Let $\M=(U,\covectors)$ be a COM of VC-dimension $d$. For a sample $S\in
\Sam(\M)$ of $\M$, consider the tope $T'=S\setminus S^0$ of
$\M':=\M\setminus S^0$ and let $X'$ be the lexicographically minimal lower circuit for $S$, i.e.,
the lexicographically minimal support-minimal covector of $\M'$ such that $T'\geq X'$.
Denote by $\M'(X')=\M'\setminus\underline{X'}$ the simple OM defined by the face
$\face(X')$ of $\M'$. Define $\alpha(S)_e=S_e$ if  $e\in f_{\M'(X')}(T')$ and  $\alpha(S)_e=0$ otherwise.
The map $\alpha$ is well-defined since $T'$ is a tope of $\M'(X')$ and hence
the sample $T'$ is full in $\M'$. Moreover, by definition we have
$\alpha(S)\leq S$, whence $\alpha(S)\in\Sam(\M)$.
Finally, by Lemma~\ref{lem:rankfunction} the $\OM$ $\M'(X')$ has VC-dimension
at most $d$ and thus, by Lemma~\ref{lem:distinguish} $\alpha(S)$ has support of
size $\le d$.

\medskip\noindent
\textbf{Reconstruction.}
To define $\beta$, pick any $C\in\{\pm 1,0\}^U$ in the image $\Ima(\alpha)$
of $\alpha$ and let $D:=\underline{C}$. Let $X$ be any covector from ${\mathcal
H}_D$, i.e., $X$ is a covector of $\covectors$ that maximally shatters $D$.
By Lemma~\ref{lem:HS=HD}, $X$ exists. Let $\Omega(C,D)$ be the set of all full samples  $Q\in \Sam_f(\M(X))$ of the OM $\M(X)$ such that
$f_{\M(X)}(Q)=D$ and $Q_{|f_{\M(X)}(Q)}=C$. Lemma~\ref{beta-claim1} below shows  that $\Omega(C,D)$ is nonempty.  Thus
$\Omega(C,D)$ is an equivalence class of $(\Sam_f(\M(X)),\sim)$. By Lemma~\ref{lem:distinguish}(iii), the realizer ${\mathcal R} (C,D)=\bigcap_{Q\in \Omega(C,D)} [Q]$ of $\Omega(C,D)$
is a nonempty convex subgraph of $G(\M(X))$.  Then, let $\beta(C)$ be any tope $\ttT$ of $\M$ of the form $\ttT=\ttT_0\times X_{|\uX}$, where $\ttT_0$ is any tope from ${\mathcal R} (C,D)$.

\medskip\noindent
\textbf{Correctness.}
We prove that $(\alpha,\beta)$ defines a proper labeled sample compression
scheme, namely, we show that for all samples  $S\in\Sam(\M)$, we have (1) $\alpha(S)\leq S$ and $\alpha(S)$ has support of
size $\le d$ and (2) $\beta(\alpha(S))$ is well-defined and $\beta(\alpha(S))\geq S$. The assertion (1) has been already established.
Let $C=\alpha(S)$ and $D=\underline{C}$. To prove that $\beta$ is well-defined, we have to show  that $\Omega(C,D)$ is nonempty.
This follows from the following result:

\begin{lemma} \label{beta-claim1} $\hhS=\hS\setminus \uX=(X\circ S)\setminus \uX$ belongs to $\Omega(C,D)$.
\end{lemma}

\begin{proof} By Lemma~\ref{lem:convex-samples2},  $\hhS\in \Sam(\M(X))$. Since $X\in {\mathcal H}_D$,  by Lemma~\ref{lem:HS=HD}, $X$
satisfies $X\setminus S^0=X'$,  where $X'$ is the minimal covector of
$\M'=\M\setminus S^0$ chosen in the definition of $\alpha(S)$.  Since $X\setminus S^0=X'\leq T'=S\setminus S^0$,
we have $\Sep(X,S)=\varnothing$. By Lemma~\ref{lem:convex-samples2} $[\hS]=[X]\cap [S]$ is a nonempty convex subgraph
of $[X]$ and $[\hhS]$ is $U$-isomorphic to $[\hS]$. Since	$X\setminus \hS^0=X'$ and  both $\M(X),\M'(X')$ have the
same VC-dimension $|D|$, $\hhS$ is a full sample of $\M(X)$ by the last assertion of Lemma~\ref{upper0}.

By Lemma~\ref{upper1}, $\M'(X')=\M(X)\setminus \hS^0=\M(X)\setminus \hhS^0$. By definition of $\alpha$ and $f_{\M(X)}$, we have
$\alpha(S)=D=f_{\M'(X')}(S)=f_{\M(X)}(\hhS)$. 
It remains to show that $\hhS_{|D}=C_{|D}$. Pick any $e\in D$. Since $C_e\ne 0$ and $C=\alpha(S)\leq S$, we get $S_e=C_e$.
Since $D\subseteq X^0$ and $\hhS=(X\circ S)\setminus \uX$, we conclude that $\hhS_e=C_e$, establishing the equality $\hhS_{|D}=C_{|D}$. This shows that
 $\hhS$ indeed belongs to  $\Omega(C,D)$.
\end{proof}

It remains to prove that $\beta(\alpha(S))\geq S$.
Since  Lemma~\ref{lem:convex-samples2} implies $[\hS]=[X]\cap [S]$, we conclude that
$\Sep(X,S)=\varnothing$ and consequently that $\hS=X\circ S \geq S$ holds. 
By definition, $\beta(\alpha(S))=\beta(C)$ is any tope of the form $\ttT=\ttT_0\times X_{|\uX}$ for a tope $\ttT_0$ of $\M(X)$ belonging to
the realizer ${\mathcal R} (C,D)=\bigcap_{Q\in \Omega(C,D)} [Q]$. Since
by Lemma ~\ref{beta-claim1}, the sample $ \hhS$ belongs to $ \Omega(C,D)$, the realizer ${\mathcal R}(C,D)$ is included in $[\hhS]$.
Consequently, $\ttT_0\geq \hhS$. Since $\hS=\hhS\times X_{|\uX}$  and $\ttT=\ttT_0\times X_{|\uX}$, we deduce that $\ttT\ge \hS$. Since $\hS\ge S$,
we obtain $\beta(\alpha(S))=\ttT\geq S$.
%
%
%
This concludes the proof of Theorem~\ref{main}, the main result of the paper.
%

\begin{remark}
 Note that by Lemma~\ref{beta-claim1} any tope $T\ge \hS$ or any tope of the
 form $\ttT=\ttT_0\times X_{|\uX}$ for a tope $\ttT_0\geq \hhS$ would be feasible. However, $S$
 and henceforth $\hS$ and $\hhS$ are not known to the reconstructor. Thus, we
 have to rely on the realizer $\mathcal{R}(C,D)\subseteq[\hhS]$.
\end{remark}


We conclude this section with two examples illustrating our compression scheme:

\begin{example}\label{ex:compression}
	Consider the tope graph $G$ of a COM $\M$ of VC-dimension $3$ and a
	realizable sample $S = (++-0-0+0)$ in Figure~\ref{fig:localization_lemma}.
	$[S]$ is induced by 7 topes drawn as white vertices of $G$. Contracting the
	3 dashed $\Theta$-classes corresponding to $\{4,6,8\}=S^0$, yields the tope
	graph $G'$ of $\M'=\M\setminus S^0$. Then $T'=S\setminus S^0= (++--+)$. The
	compressor picks $X'= (0+--+)$, the lexicographically minimal lower circuit
	for $S$; $X'$ corresponds to the thick red edge in $G'$, and in covector representation $\M'(X')=(\{1\},\{(0),(+),(-)\})$.
	The compressor returns $\alpha(S) =	(+0000000)$ and $D=\{1\}$.
	The reconstructor receives $C=(+0000000)=\alpha(S)$, defines $D
	=\underline{C}=\{1\}$ and constructs the set $\mathcal{H}_D$. There are six
	covectors of $\M$ belonging to $\mathcal{H}_D$ corresponding to the thick
	red edges in $G$. By the localization lemma, they are the covectors which
	have the same VC-dimension as $X'$ and agree with $X'$ on
	$\{1,2,3,5,7\}=\underline{S}$. The reconstructor picks an arbitrary
	covector from $\mathcal{H}_D$, say $X = (0+----+-)$. The OM $\M(X)$ is
	composed of the covectors $X$ and the
	ends $T_1$ and $T_2$ of the corresponding red edge. Then, we get $\Omega(C,D) = \{T_1\}$
    and its realizer is ${\mathcal R}(C,D)= [T_1]$.
	Thus, $\beta(\alpha(S))$ is set to $T_1$, which is a white vertex
	of $G$.
\end{example}
\begin{figure}[htb]
	\centering
	\includegraphics[width=0.74\linewidth]{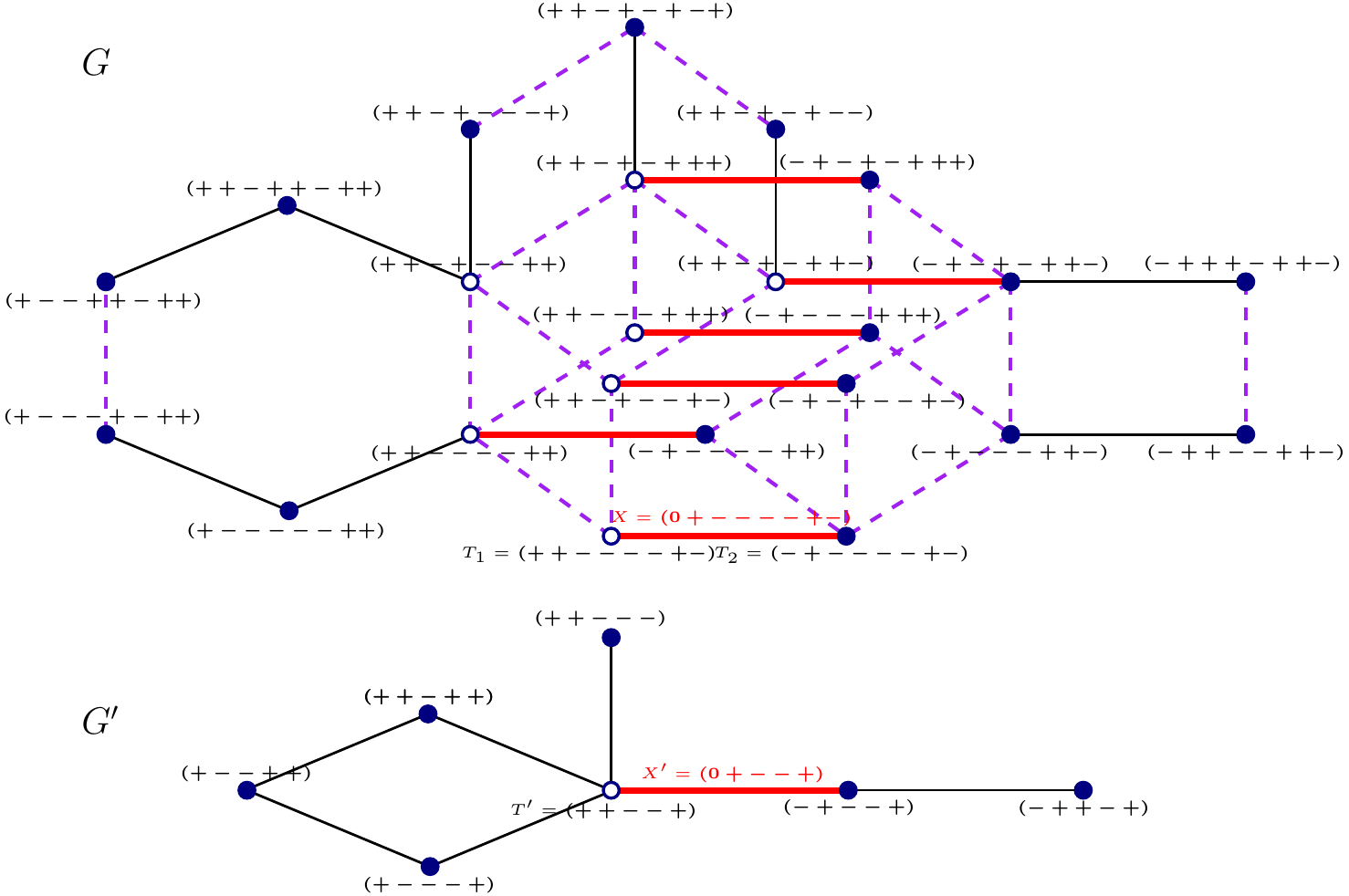}
	\caption{
		\label{fig:localization_lemma} An illustration of
		Example~\ref{ex:compression}.
	}
\end{figure}

 The previous example might suggest that indeed $f_{\M}(S)=f_{\M}(S')$ and
 $S_{|f_{\M}(S)}=S'_{|f_{\M}(S')}$ together imply $[S]=[S']$. However, the next example shows that
 $[S]$ and $[S']$ might not even be contained in each other.

\begin{example}
Let $\M$ be the COM whose tope graph consists of a 4-cycle $C$ with two edges
pending on the same vertex $T$. Let $1,2$ be the $\Theta$-classes of $C$ and
$3,4$ the $\Theta$-classes of the other two edges. Each of the two pending edges
corresponds to a full sample $S=(+++0)$ and $S'=(++0+)$ respectively. It is
easy to see that $f_{\M}(S)=f_{\M}(S')=\{1,2\}$ and
$S_{|f_{\M}(S)}=S'_{|f_{\M}(S')}=(++)$ but $[S]\cap [S']=\{ T\}$. Further note that the tope
graph of $\M$ can be easily embedded into a tope graph of a uniform OM $\M'$ of
rank 3 in which $C$ is a cocircuit, the samples $S,S'$ encode the two
pending edges (with possibly larger support) and still
$f_{\M'}(S)=f_{\M'}(S')$ and $S_{|f_{\M'}(S)}=S'_{|f_{\M'}(S')}$ while
$[S]\cap [S']=\{ T\}$ is a proper subset of both $[S]$ and $[S']$.
\end{example}

\section{Conclusion}
We have presented proper labeled compression schemes of size $d$ for COMs of
VC-dimension $d$. This is a generalization of the results of~\cite{MoWa} for ample set systems,
of~\cite{BDLi} for affine arrangements of hyperplanes, and of our result~\cite{ChKnPh_CUOM} for
complexes of uniform oriented matroids.  Even though we made strong use of the structure of COMs, it
is tempting to extend our approach to other classes, e.g., bouquets of oriented
matroids~\cite{DeFu}, strong elimination systems~\cite{BaChKn}, or CW
left-regular-bands~\cite{Margolis}. Our treatment of realizable samples as
convex subgraphs suggests an angle at general partial cubes.

Our results together with the approach  of~\cite{ChKnPh_2d,ChKnPh_CUOM} suggest a new approach 
at \emph{improper} labeled compression schemes of COMs. For this one needs to answer the 
question: Is it possible to extend a given set system or a partial cube
to a COM without increasing the VC-dimension too much?

In unlabeled sample compression schemes, the compressor $\alpha$
is less expressive since its image is in $2^U$ and has to satisfy
$\alpha(S)\subseteq\underline{S}$. Unlabeled compression
schemes exist for realizable affine oriented matroids~\cite{BDLi} and ample set systems with corner peelings~\cite{ChChMoWa,KuWa}.
Recently, Marc~\cite{Ma_Om} designed unlabeled sample compression schemes for OMs. His construction uses Oriented Matroid Programming and Lemma~\ref{lem:distinguish}. Moreover, he shows there are unlabeled
compression schemes for COMs with corner peelings -- a recent notion introduced in
~\cite{KnMa1}.


\section*{Acknowledgments} We acknowledge the referees for their careful reading of the manuscript
and useful suggestions and comments. Especially, we are very grateful to one referee who found a critical error in a previous version of the distinguishing lemma. We are also grateful to them for a new proof of  Lemma~\ref{lem:shatterandcircuits}
and to Emeric Gioan for finding a gap in a previous proof of this lemma.
This research was supported  by ANR project DISTANCIA (ANR-17-CE40-0015).

\bibliographystyle{plain}
%

\end{document}